\documentclass[10pt, oneside, reqno]{amsart}

\makeatletter
\@namedef{subjclassname@2020}{%
  \textup{2020} Mathematics Subject Classification}
\makeatother

\usepackage{amsfonts, amssymb, amsmath, graphicx}
\usepackage[all, knot]{xy}
\usepackage{multicol}

\newcommand{\N}{\mathbb{N}}

\newcommand{\cM}{{\mathcal M}}
\newcommand{\cG}{{\mathcal G}}
\newcommand{\btau}{\bar{\tau}}
\newcommand{\bgamma}{\bar{\gamma}}

\newtheorem{theorem}{Theorem}
\theoremstyle{definition}
\newtheorem{corollary}[theorem]{Corollary}
\newtheorem{lemma}[theorem]{Lemma}
\newtheorem{proposition}[theorem]{Proposition}

\newtheorem{definition}[theorem]{Definition}

\numberwithin{equation}{section}

\def\co{\colon\thinspace}

\begin{document}
\title{Algebraic structures among virtual singular braids}

\author{Carmen Caprau}
\address{Department of Mathematics, California State University, Fresno;  5245 N. Backer Avenue, M/S PB108, CA 93740, USA}
\email{ccaprau@csufresno.edu}

\author{Antonia Yeung}
\address{Department of Mathematics, California State University, Fresno;  5245 N. Backer Avenue, M/S PB108, CA 93740, USA}\email{antoniayeung@csufresno.edu}

\thanks{This research was partially supported by Simons Foundation grant $\#355640$ - Carmen Caprau}

\begin{abstract}
We show that the virtual singular braid monoid on $n$ strands embeds in a group $VSG_n$, which we call the virtual singular braid group on $n$ strands. The group $VSG_n$ contains a normal subgroup $VSPG_n$ of virtual singular pure braids.  We show that $VSG_n$ is a semi-direct product of  $VSPG_n$ and the symmetric group $S_n$. We provide a presentation for $VSPG_n$ via generators and relations. We also represent $VSPG_n$ as a semi-direct product of $n-1$ subgroups and study the structures of these subgroups. These results yield a normal form of words in the virtual singular braid group. 
\end{abstract}

\maketitle

\section{Introduction} \label{intro} 
We can study classical knots by studying the algebraic structures of classical braids. Due to a theorem by Markov~\cite{Markov}, the classification of knots and links is equivalent to certain algebraic properties of classical braids. Similarly, we can study algebraic structures of virtual braids, singular braids, virtual singular braids, and welded braids to classify virtual knots, singular knots, virtual singular knots, and welded knots, respectively (see, for example~\cite{Bardakov, BB, CPM, Caprau, Fenn, G, Ka, KL, KL2}).

In this paper we study some of the algebraic properties of virtual singular braids. Virtual singular braids are similar to classical braids, with the difference that they contain virtual and singular crossings, besides classical crossings. Two virtual singular braids $\alpha$ and $\beta$ on $n$ strands are \textit{multiplied} (or \textit{composed}) using vertical concatenation. The braid $\alpha \beta$ is formed by placing $\alpha$ on top of $\beta$ and gluing the bottom endpoints of $\alpha$ with the top endpoints of $\beta$. Under this binary operation, the set of isotopy classes of virtual singular braids on $n$ strands forms a monoid. This monoid can be defined as follows.

\begin{definition} \label{def:vsbn} 
Let $n \in \N$, $n \geq 2$. The \textit{virtual singular braid monoid on $n$ strands}, ${VSB}_n$, is the monoid generated by the \textit{elementary virtual singular braids} $\{\sigma_i, \sigma_i^{-1}, v_i, \tau_i, \, | \, 1 \leq i \leq n-1\}$:

\[  \sigma_i \,\,\, =\,\,\,  \raisebox{-17pt}{\includegraphics[height=.5in]{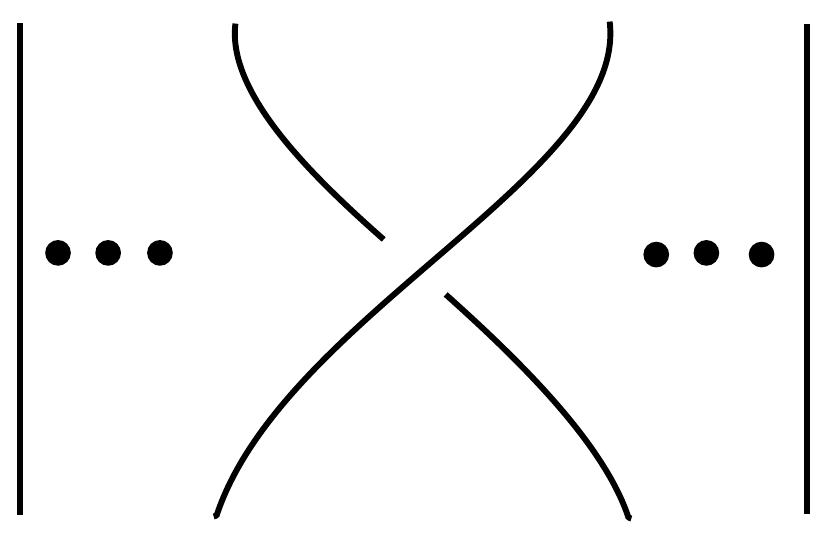}} \hspace{1cm} \sigma_i^{-1} \,\,\,=\,\,\, \raisebox{-17pt}{\includegraphics[height=.5in]{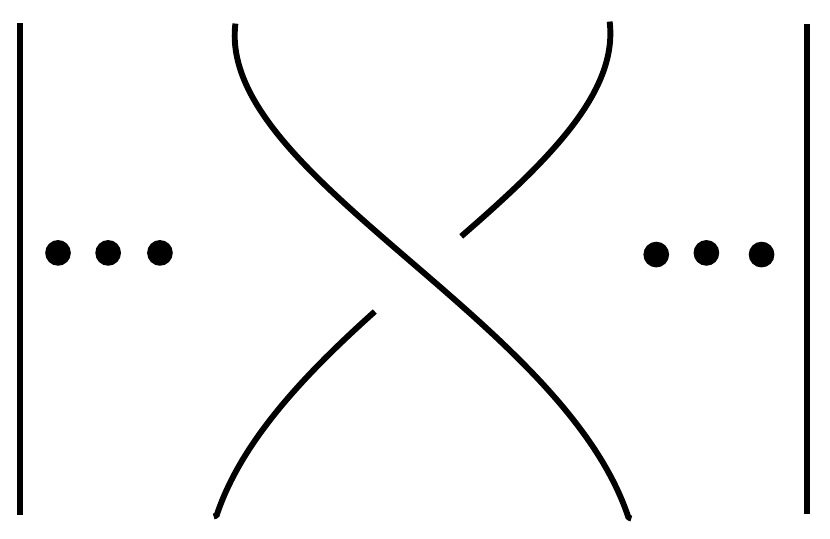}}
\put(-180, 21){\fontsize{7}{7}$1$}
\put(-165, 21){\fontsize{7}{7}$i$}
\put(-150, 21){\fontsize{7}{7}$i+1$}
\put(-127,21){\fontsize{7}{7}$n$}
\put(-58, 21){\fontsize{7}{7}$1$}
\put(-40, 21){\fontsize{7}{7}$i$}
\put(-25, 21){\fontsize{7}{7}$i+1$}
\put(-3,21){\fontsize{7}{7}$n$}
\]
\vspace{0.2cm}
\[ \tau_i \,\,\,= \,\,\, \stackrel \,\, \raisebox{-17pt}{\includegraphics[height=.5in]{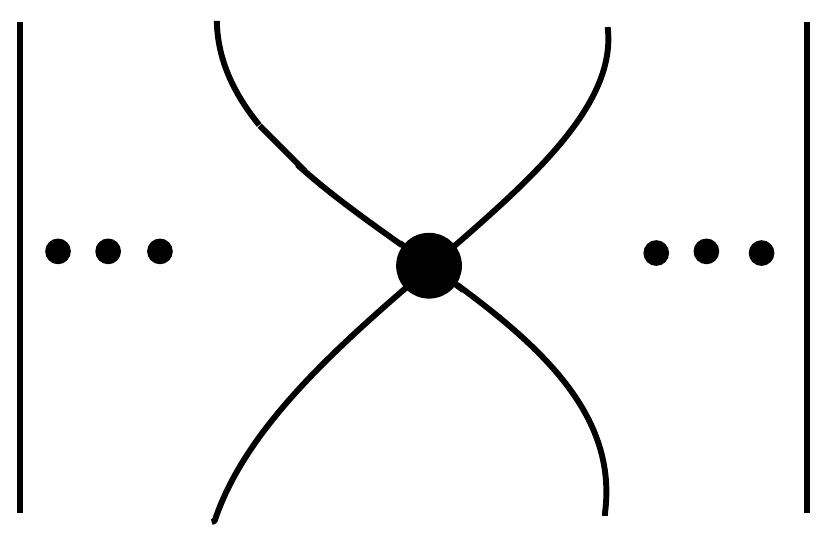}} \hspace{1cm} v_i \,\,\, = \,\,\, \raisebox{-17pt}{\includegraphics[height=.5in]{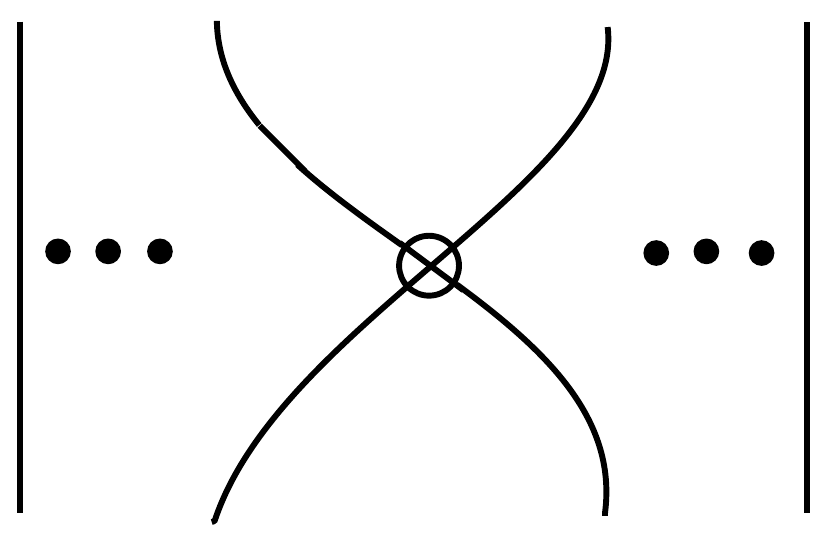}}
\put(-171, 21){\fontsize{7}{7}$1$}
\put(-158, 21){\fontsize{7}{7}$i$}
\put(-140, 21){\fontsize{7}{7}$i+1$}
\put(-119,21){\fontsize{7}{7}$n$}
\put(-56, 21){\fontsize{7}{7}$1$}
\put(-42, 21){\fontsize{7}{7}$i$}
\put(-25, 21){\fontsize{7}{7}$i+1$}
\put(-3,21){\fontsize{7}{7}$n$}
\]
with the defining relations: 

\begin{enumerate}
\item $\sigma_i\sigma_i^{-1}=\sigma_i^{-1}\sigma_i=1_n$
\begin{eqnarray*}
\raisebox{-.7cm}{\includegraphics[height=0.7in]{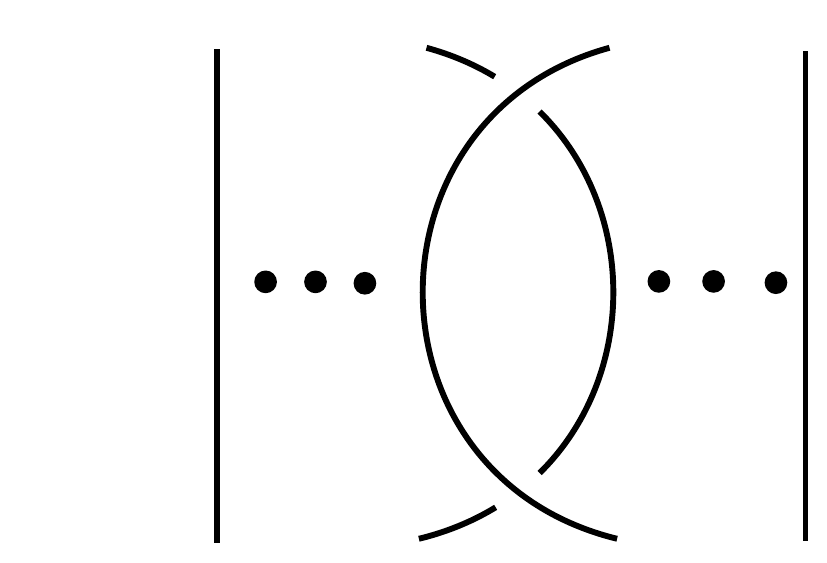}} \hspace{0.5cm} \sim \hspace{0.18cm} &\raisebox{-1.2cm}{\includegraphics[height=0.85in]{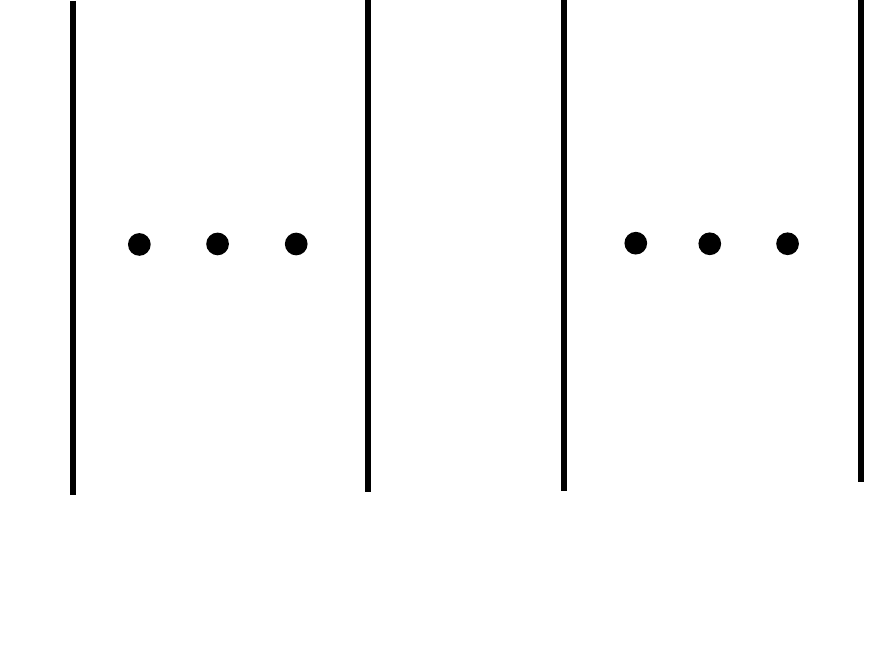}}
\end{eqnarray*}

\item $v_i^2=1_n$
\begin{eqnarray*}
\raisebox{-.7cm}{\includegraphics[height=0.7in]{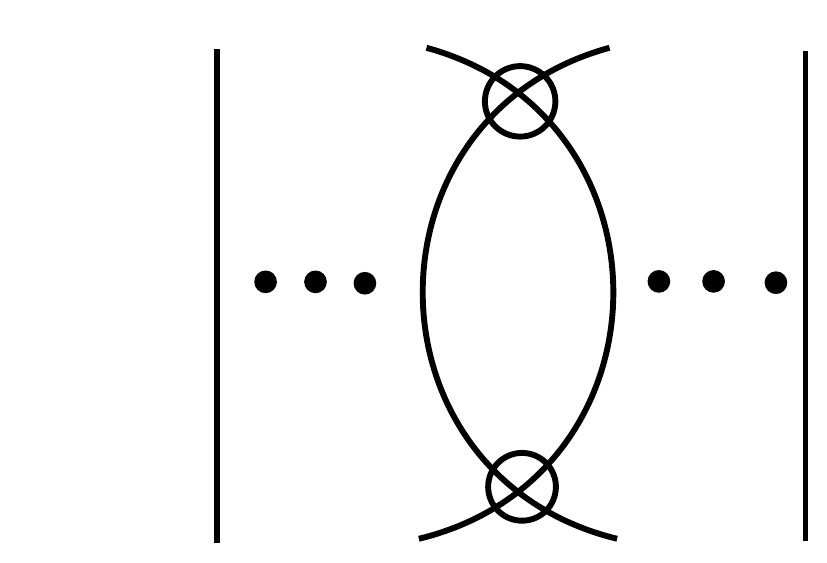}} \hspace{0.5cm} \sim \hspace{0.18cm} &\raisebox{-1.2cm}{\includegraphics[height=0.85in]{idb}}
\end{eqnarray*}

\item $\sigma_i\sigma_j\sigma_i=\sigma_j\sigma_i\sigma_j$, for $|i-j|=1$
\begin{eqnarray*}
\raisebox{-.7cm}{\includegraphics[height=0.7in]{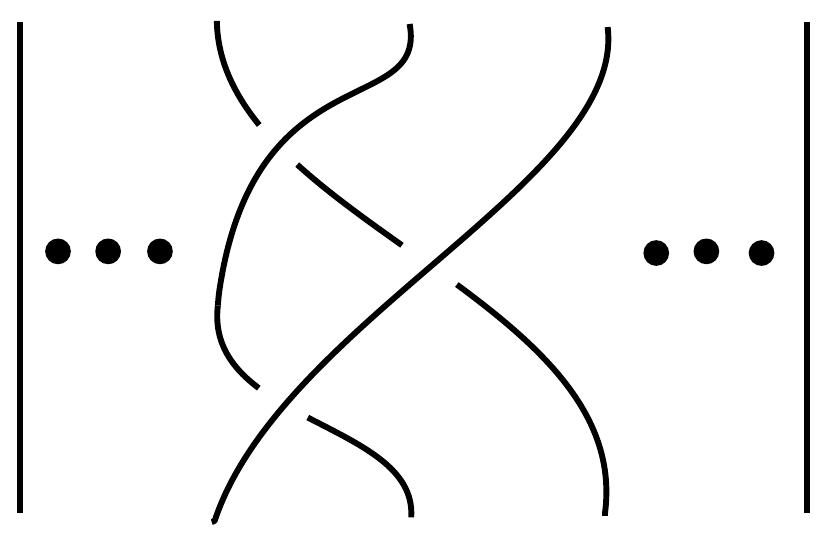}} \hspace{0.5cm} \sim \hspace{0.2cm} &\raisebox{-.7cm}{\includegraphics[height=0.7in]{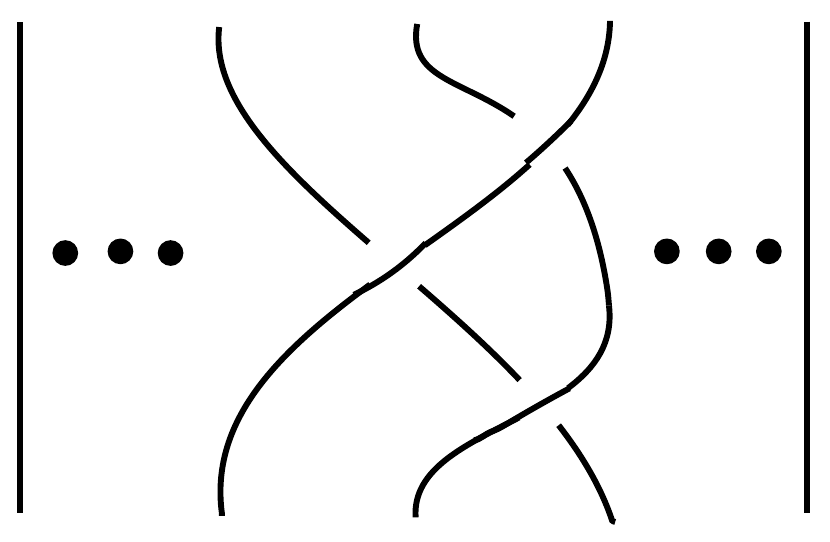}}
\end{eqnarray*}

\item $v_{i}v_{j}v_{i}=v_{j}v_{i}v_{j}$, for $|i-j|=1$
\begin{eqnarray*}
\raisebox{-.7cm}{\includegraphics[height=0.7in]{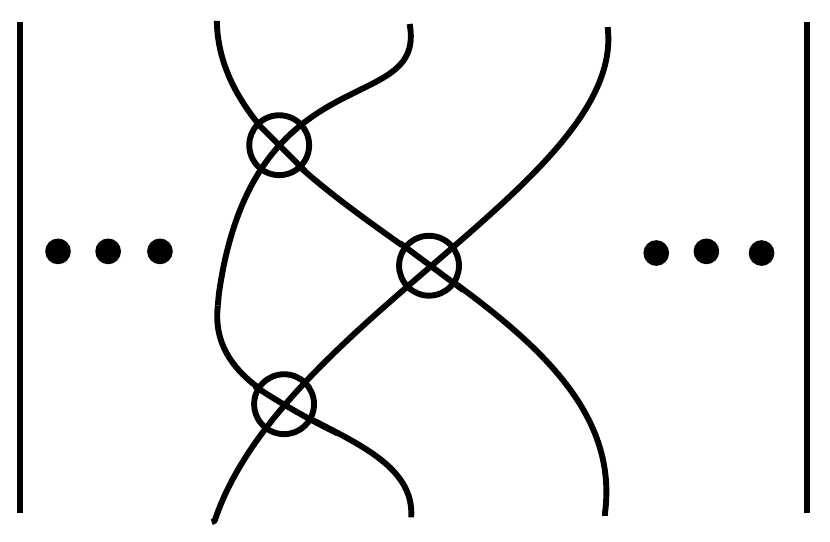}} \hspace{0.5cm} \sim \hspace{0.2cm}&\raisebox{-.7cm}{\includegraphics[height=0.7in]{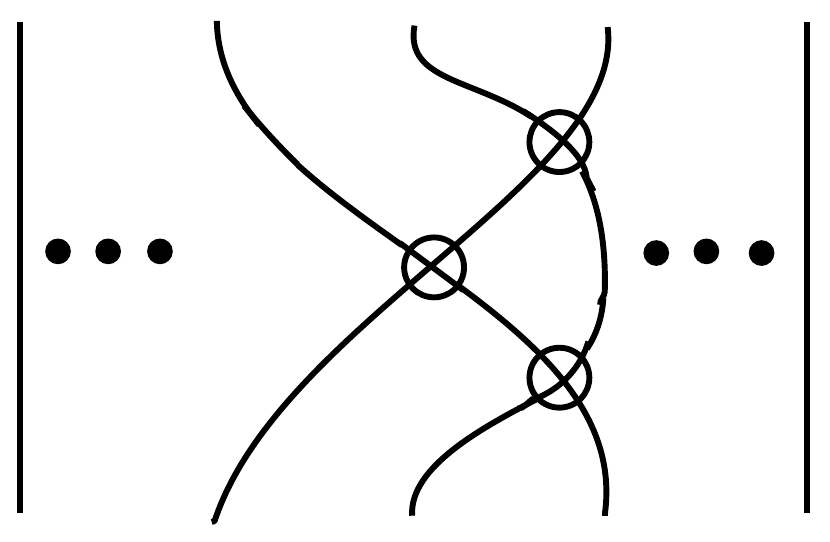}}
\end{eqnarray*}

\item $v_i\sigma_{j}v_i=v_{j}\sigma_iv_{j}$, for $|i-j|=1$ 
\begin{eqnarray*}
\raisebox{-.7cm}{\includegraphics[height=0.7in]{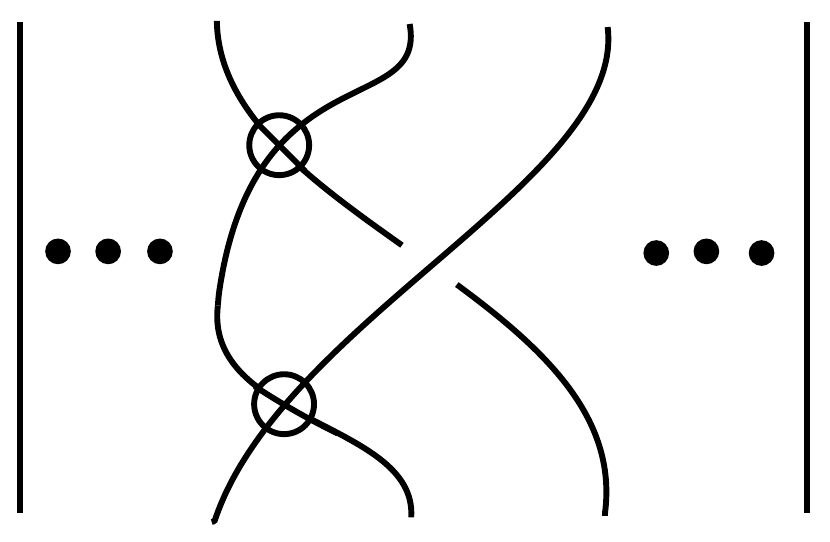}} \hspace{0.5cm} \sim \hspace{0.2cm} &\raisebox{-.7cm}{\includegraphics[height=0.7in]{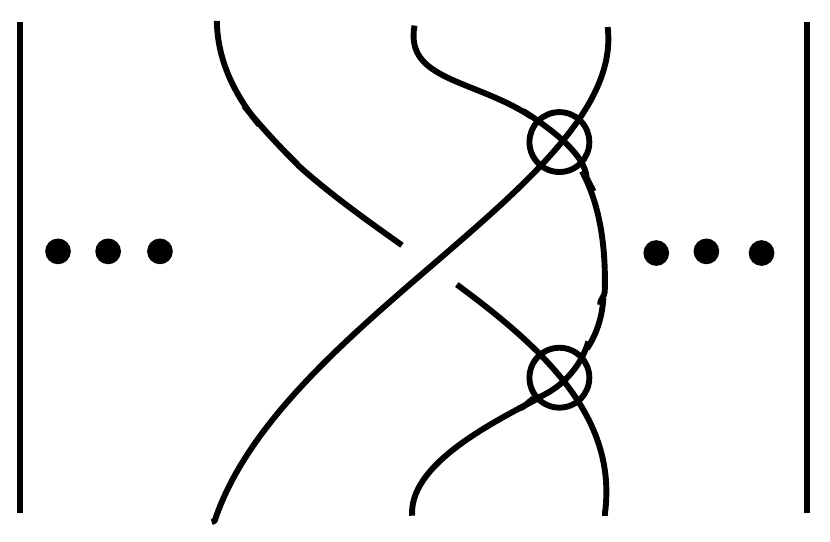}}
\end{eqnarray*}

\item $v_i\tau_{j}v_i=v_{j}\tau_iv_{j}$, for $|i-j|=1$
\begin{eqnarray*}
\raisebox{-.7cm}{\includegraphics[height=0.7in]{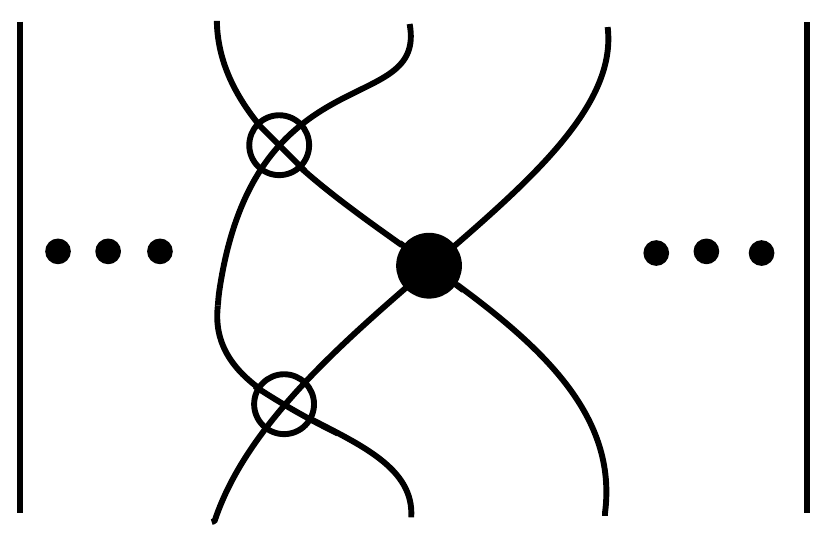}} \hspace{0.5cm} \sim \hspace{0.2cm} &\raisebox{-.7cm}{\includegraphics[height=0.7in]{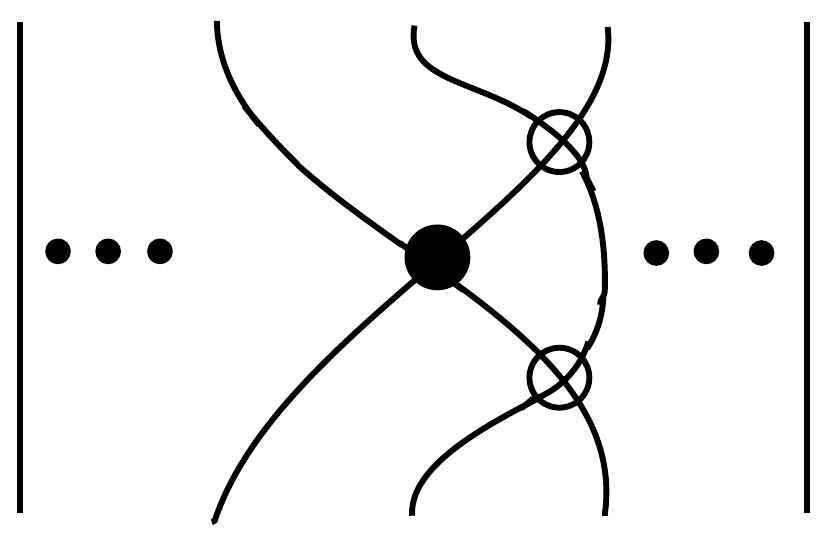}}
\end{eqnarray*}

\item $\sigma_i\sigma_j\tau_i=\tau_j\sigma_i\sigma_j$, for $|i-j|=1$.
\begin{eqnarray*}
\raisebox{-.7cm}{\includegraphics[height=0.7in]{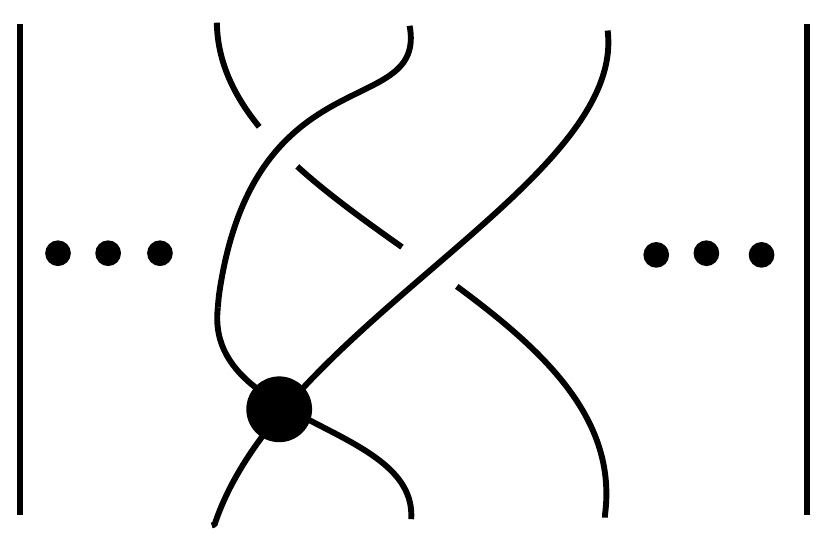}} \hspace{0.5cm} \sim \hspace{0.2cm} &\raisebox{-.7cm}{\includegraphics[height=0.7in]{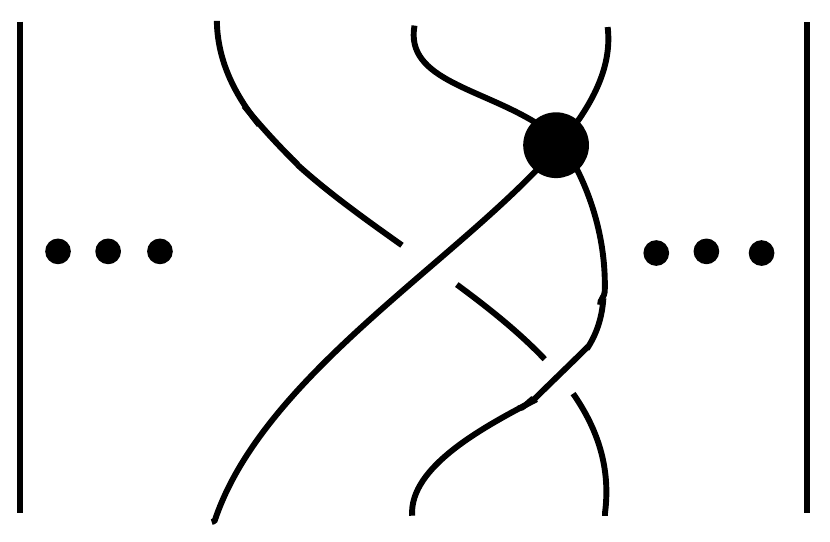}}
\end{eqnarray*}

\item $\sigma_i\tau_i=\tau_i\sigma_i$
\begin{eqnarray*}
\raisebox{-.7cm}{\includegraphics[height=0.7in]{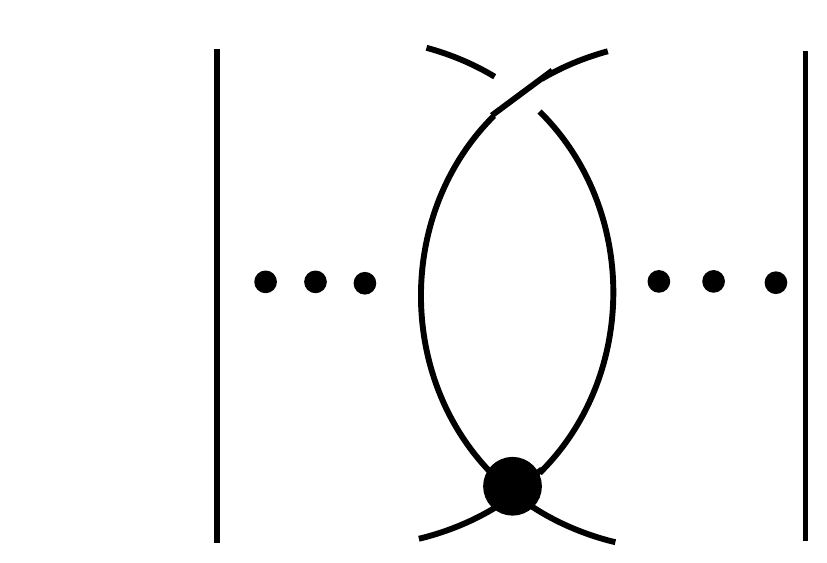}} \hspace{0.5cm} \sim  \hspace{-0.3cm} &\raisebox{-.7cm}{\includegraphics[height=0.7in]{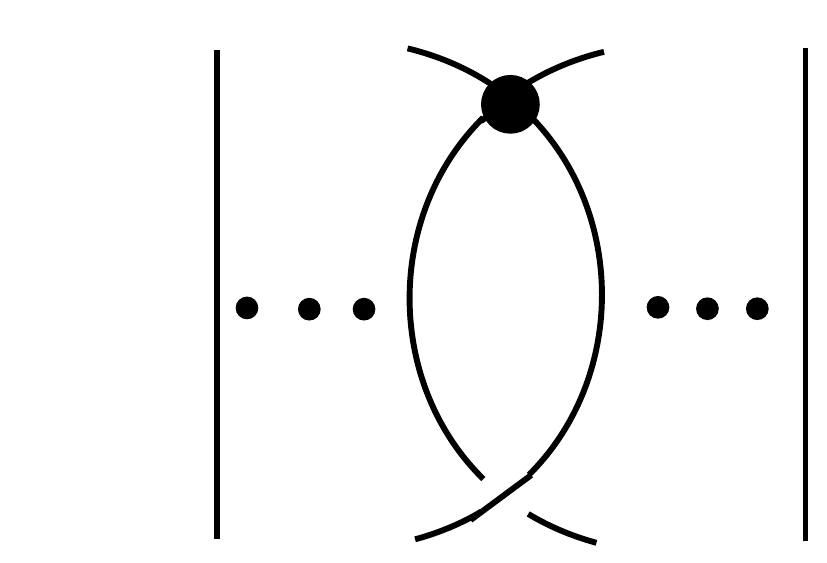}}
\end{eqnarray*}

\item $g_ih_j = h_jg_i $ for $|i-j| >1$, where $g_i, h_i \in \{ \sigma_i, \tau_i, v_i  \}$.
\begin{eqnarray*}
\raisebox{-.7cm}{\includegraphics[height=0.7in]{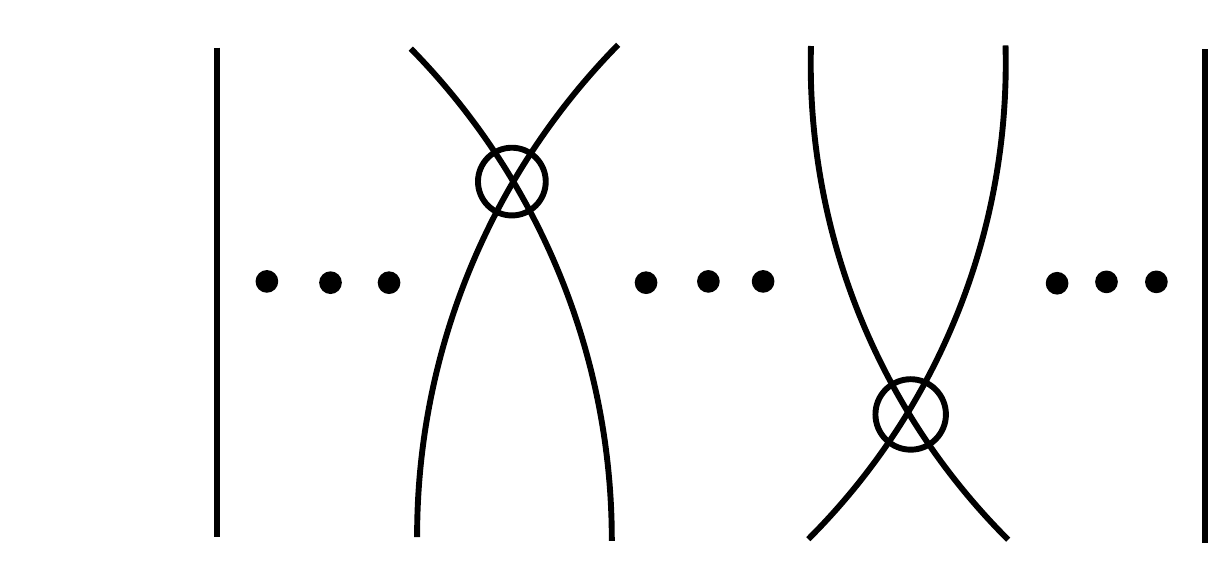}} \hspace{0.5cm} \sim \hspace{-.3cm} &\raisebox{-.7cm}{\includegraphics[height=0.7in]{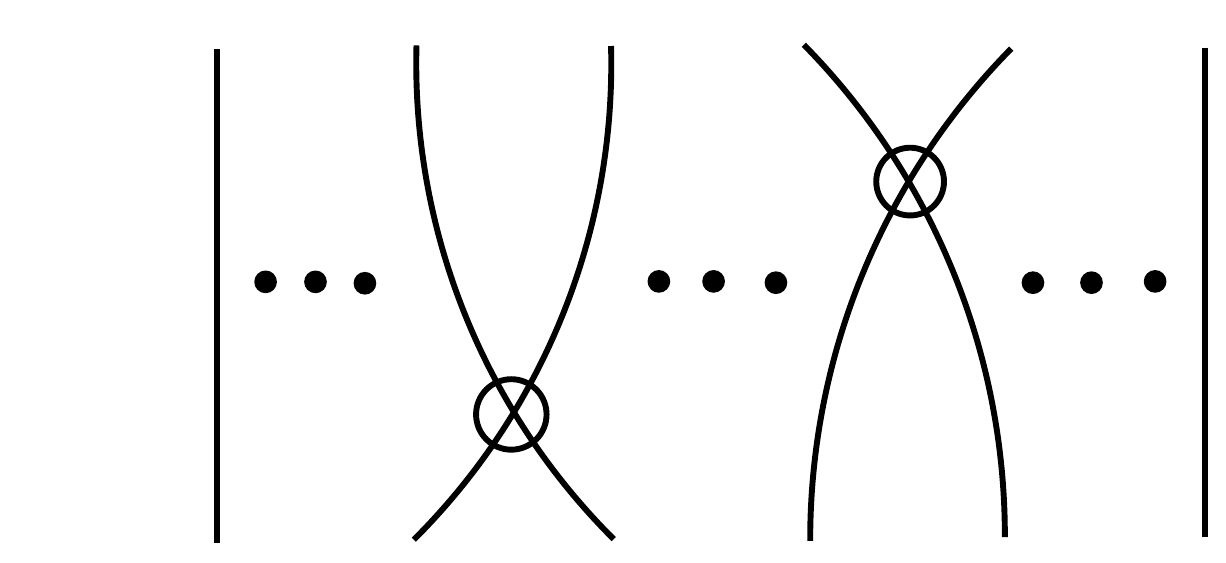}}
\end{eqnarray*}
\end{enumerate}
\end{definition}

Notice that some relations can be written in different equivalent forms. For example, relation (6) is equivalent to $v_jv_i\tau_j=\tau_iv_jv_i$, where $|i-j|=1$. We also remark that relation (7) says that we can slide a strand under or over a singular crossing; we only depicted the relation representing the sliding of a strand under a singular crossing.

The relations in Definition~\ref{def:vsbn} define virtual singular braid isotopy. Each of the defining relations for $VSB_n$ is a braided version of Reidemeister-type moves for diagrams of virtual singular knots and links, with the exception of the type I moves (involving classical and virtual crossings) which cannot be represented in braid form. The generators $\sigma_i$ and $v_i$ are invertible (due to relations (1) and (2) in Definition~\ref{def:vsbn}), with inverses $\sigma_i^{-1}$ and  respectively $v_i$, while the generators $\tau_i$ are not invertible (since there is no type II move for virtual singular link diagrams involving singular crossings).

In this paper we show that the virtual singular braid monoid $VSB_n$ embeds in a group, which we refer to as the \textit{virtual singular braid group} and we denote it by $VSG_n$. This group contains a normal subgroup, $VSPG_n$, which we refer to as the \textit{virtual singular pure braid group}. We show that $VSG_n/VSPG_n \cong S_n$ and that $VSG_n \cong VSPG_n  \rtimes S_n$. In addition, we prove that the group $VSPG_n$ is representable as the semi-direct product $VSPG_n = VS_{n-1}^* \rtimes \left(VS_{n-2}^* \rtimes \left( \cdots \rtimes \left(VS_2^* \rtimes VS_1^* \right) \cdots  \right) \right)$,
where $VS_i^*$, for all $2 \le i \le n-1$, are infinitely generated subgroups of $VSPG_n$ and $VS_1^*$ is a subgroup of rank $4$. As a consequence of these results, we obtain a normal form of words in $VSPG_n$ and $VSG_n$.

The paper is organized as follows. Section~\ref{embedding} is dedicated to proving that the virtual singular braid monoid, $VSB_n$, embeds in a group, $VSG_n$, and we describe this group. In Section~\ref{another presentation} we provide a second presentation for the group $VSG_n$; this presentation has as generators certain type of virtual singular pure braids, called elementary fusing strings, together with the virtual generators for $VSB_n$. There is a normal subgroup of $VSG_n$, denoted by $VSPG_n$, whose elements are virtual singular pure braids, and the quotient group $VSG_n/VSPG_n$ is isomorphic to the symmetric group $S_n$. While it is sufficient to work with the elementary fusing strings (together with the virtual generators) to generate the group $VSB_n$, this is not the case for the subgroup $VSPG_n$, for which we need more general virtual singular pure braids, which we refer to as generalized fusing strings. In Section~\ref{presentation VSPGn} we use the Reidemeister-Schreier method to find a presentation for the normal subgroup $VSPG_n$ using generators and relations. In Section~\ref{semi-direct product} we prove that the virtual singular pure braid group $VSPG_n$ decomposes as a semi-direct product of $n-1$ subgroups and study the structures of these subgroups. We also show how this decomposition provides a normal form of words in $VSPG_n$, and hence in $VSG_n$.

\section{Virtual singular braid monoid embeds in a group}\label{embedding}

It was shown by Fenn, Keyman and Rourke in~\cite{Fenn} that the singular braid monoid $SB_n$ embeds in a group, denoted by $SG_n$ and called the \textit{singular braid group}. The elements of this group have a geometric interpretation as singular braids with two types of singular crossings that cancel: $\tau_i$ and $\bar{\tau_i}$. We will adopt the representation of the new type of singularity as an open blob, as shown below:
\[\btau_i =  \raisebox{-0.8cm}{\includegraphics[height=0.7in]{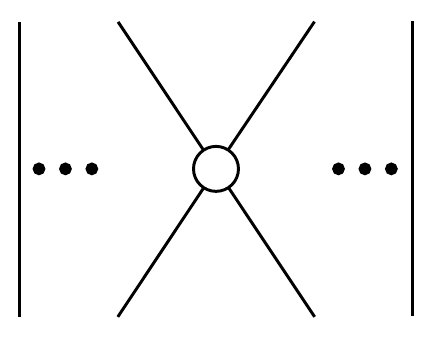}} \put(-53, 27){\fontsize{7}{7} $i$ \quad $i+1$} \]

The method used in~\cite{Fenn} is purely geometric. Motivated by the result in~\cite{Fenn}, Keyman~\cite{Keyman} developed a method that uses elementary algebraic properties to show that certain type of monoids embed in groups; these monoids have  presentations similar to that of the singular braid monoid. 

We denote the free monoid on a set $A$ by $F^+(A)$. The following theorem was proved in~\cite{Keyman}, and thus we omit it's proof here.

\begin{theorem} \label{Monoid Embed in a Group}
\cite[Theorem 3]{Keyman}
Let $\cM$ be a monoid given by a presentation $[A \cup B \, | \,{\bf R}]$ where $A=\{ a_1, \dots, a_n \}$, $B=\{ b_1, \dots, b_m \}$ and ${\bf R  = R_1 \cup R_2 \cup  R_3 \cup R_4}$, where:
\begin{itemize}
\item ${\bf R_1}$ consists of relations of the form $u=v$, where $u,v \in F^+(A)$;
\item ${\bf R_2} =\{a_iu_i=u_ia_i=1 | \text{ for some } u_i \in F^+(A), \text{ for all } i=1, \dots, n\}$;
\item ${\bf R_3}$ consists of relations of the form $ub_j=b_ku$, for some $j,k=1, \dots, m$ and $u \in F^+(A)$;
\item ${\bf R_4}$ consists of relations of the form $b_jb_k=b_kb_j$, for some $j,k=1, \dots, m$.
\end{itemize}
Then $\cM$ embeds in a group $\cG$ with presentation $[A \cup B \cup \bar{B} \, | \, {\bf R \cup R' }]$, where $\bar{B}=\{ \bar{b}_j | j = 1, \dots, m  \}$ and ${\bf R'}$ consists of the following relations:

$b_j\bar{b}_j=1=\bar{b}_jb_j$,

$u\bar{b}_j=\bar{b}_ku$ if $ub_j=b_ku \in {\bf R}$, and

$\bar{b}_jb_k=b_k\bar{b}_j$ and $\bar{b}_j\bar{b}_k=\bar{b}_k\bar{b}_j$ if $b_jb_k=b_kb_j \in {\bf R}$.
\end{theorem}

In this presentation of a monoid $\cM$, the invertible elements in $A$ can satisfy any relations among them. The elements in $B$ have no inverses (neither left nor right). If $B = \emptyset$ then $\cM$ is given by the presentation $[A \, | \, {\bf R_1\cup R_2}]$ and is a group. 

We show that the virtual singular braid monoid $VSB_n$ embeds in a group using the above theorem. This group is defined as follows.

\begin{definition} \label{VSGn}
Let $n \in \N, n \ge 2$. The \textit{virtual singular braid group}, $VSG_n$, is defined as the group  generated by the same generators as $VSB_n$ along with $\btau_i$, where $1 \le i \le n-1$. The defining relations consist of: 
\begin{enumerate}
\item[(i)] The same monoid relations as $VSB_n$ with additional relations obtained by substituting $\btau_i$ for $\tau_i$ in each relevant relation, namely

$\sigma_i \btau_i = \btau_i\sigma_i$, for all $1 \le i \le n-1$;

$\sigma_i\sigma_j\btau_i = \btau_j\sigma_i\sigma_j$, for $|i-j|=1$;

$v_iv_j\btau_i = \btau_jv_iv_j$, for $|i-j|=1$;

$\btau_i \alpha_j = \alpha_j \btau_i$, for $|i-j|>1$, where $\alpha_j \in \{ \sigma_j,v_j,\tau_j,\btau_j \}$; and

\item[(ii)] $\btau_i\tau_i = 1_n = \tau_i\btau_i$, for  all $1 \le i \le n-1$, which are depicted below:
\[\raisebox{-1cm}{\includegraphics[height=0.9in]{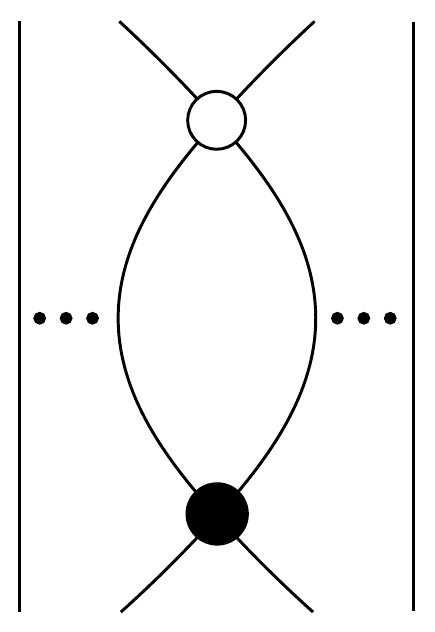}} \hspace{0.2cm} = \hspace{0.2cm} \raisebox{-1cm}{\includegraphics[height=0.9in]{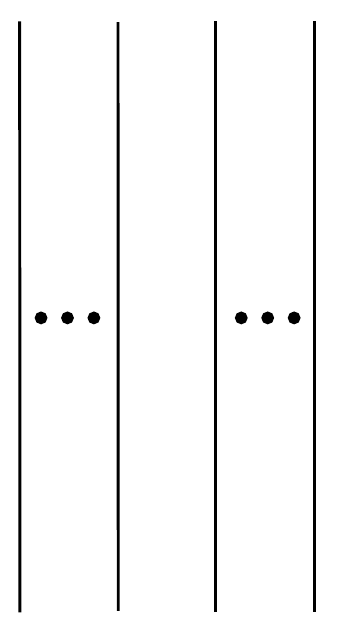}} \hspace{0.2cm} = \hspace{0.2cm} \raisebox{-1cm}{\includegraphics[height=0.9in]{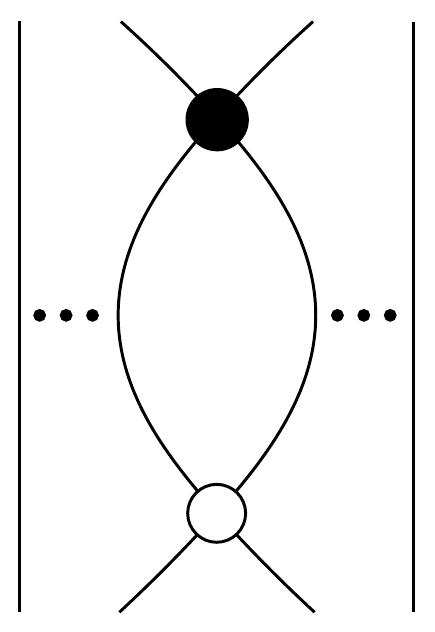}} \]
\end{enumerate}
\end{definition}
Relations (ii) in Definition~\ref{VSGn} say that $\btau_i = \tau_i^{-1}$ and make $VSG_n$ into a group. We will call an element in $VSG_n$ an \textit{extended virtual singular braid}.

\begin{theorem}
Let $n \in \N, n \ge 2$. The virtual singular braid monoid $VSB_n$ embeds in the virtual singular braid group $VSG_n$.
\end{theorem}
\begin{proof} $VSB_n$ satisfies the conditions in Theorem \ref{Monoid Embed in a Group}. Indeed, $VSB_n$ has a presentation in the form $[A \cup B \, | \, {\bf R_1 \cup R_2 \cup R_3 \cup R_4} ]$, where
\[A=\{\sigma_i, \sigma_i^{-1},  v_i \mid 1 \le i \le n \}, B=\{\tau_{i} \mid 1 \le i \le n \}, \, \text{ and} \]
\begin{itemize}
\item ${\bf R_1}$ consists of the defining relations:
	\subitem $\sigma_i\sigma_j\sigma_i=\sigma_j\sigma_i\sigma_j, \,\,\,  v_iv_jv_i=v_jv_iv_j, \,\,\, v_i\sigma_jv_i=v_j\sigma_iv_j$, for $|i-j|=1$
	\subitem $\sigma_i\sigma_j=\sigma_j\sigma_i, \,\,\, v_iv_j=v_jv_i, \,\,\, \sigma_iv_j=v_j\sigma_i$, for  for $|i-j|>1$
	\subitem and similar relations involving $\sigma_i^{-1}$.
\item ${\bf R_2}$ consists of the relations
	\subitem $v_i^2=1_n$ and $\sigma_i\sigma_i^{-1}=1_n =\sigma_i^{-1}\sigma_i$, for all $1 \le i \le n-1$.
	
\item ${\bf R_3}$ consists of the  relations 
	\subitem $\sigma_i\tau_i=\tau_i\sigma_i$,
	\subitem $\sigma_i\sigma_j\tau_i=\tau_j\sigma_i\sigma_j, \,\,\,  v_i\tau_jv_i=v_j\tau_iv_j$, for $|i-j|=1$
	\subitem $\sigma_i\tau_j=\tau_j\sigma_i, \,\,\, v_i\tau_j=\tau_jv_i$, for $|i-j|>1$
\item ${\bf R_4}$ consists of the relations:
	\subitem $\tau_i\tau_j=\tau_j\tau_i$, for $|i-j|>1$.
\end{itemize}
Moreover, by the definition of $VSG_n$, it follows that $VSB_n$ embeds in $VSG_n$. 
\end{proof}

Since $\btau_i = \tau_i^{-1}$, the relations (i) in Definition~\ref{VSGn} can be derived from the analogous monoid relations involving the generators $\tau_i$. Therefore, $VSG_n$ is the group with the same group presentation as the monoid $VSB_n$. Specifically, $VSG_n$ is the group generated by $\{\sigma_i, v_i, \tau_i \, | \, 1 \leq i \leq n-1\}$, subject to the following relations:
\begin{itemize}
\item[(1)] two-point relations: $v_i^2=1_n$  and $\sigma_i\tau_{i}=\tau_i\sigma_{i}$ for all $1 \le i \le n-1$

\item[(2)] three-point relations, for all for $|i-j|=1$: 

$\sigma_i\sigma_j\sigma_i=\sigma_j\sigma_i\sigma_j, \,\,\, v_{i}v_{j}v_{i}=v_{j}v_{i}v_{j}$

$v_i\sigma_{j}v_i=v_{j}\sigma_iv_{j}, \,\,\, v_i\tau_{j}v_i=v_{j}\tau_iv_{j}$ and  $\sigma_i\sigma_j\tau_i=\tau_j\sigma_i\sigma_j$
\item[(3)] commuting relations: $g_ih_j = h_jg_i $ for $|i-j| >1$, where $g_i, h_i \in \{ \sigma_i, \tau_i, v_i  \}$.
\end{itemize}

There is an obvious group homomorphism that associates to each virtual singular braid in $VSG_n$ a permutation in $S_n$.
That is, let $\pi \co VSG_n \to S_n$ be the map defined on generators as follow:
\[ \pi(\sigma_i) =\pi(v_i) =\pi(\tau_i) =(i, i+1), \,\,\, \text{for all} \,\, \, 1 \le i \le n-1, \]
and extend it to all elements in $VSG_n$ such that $\pi$ is a homomorphism. Since the set $\{(i, i+1) \mid 1 \le i \le n-1\}$ generates $S_n$, the map $\pi$ is surjective. We can consider $S_n$ to be a subgroup of $VSG_n$ generated by $v_i$, for all $1 \le i \le n-1$. That is, we identity $v_i$ with the transposition $(i, i+1)$.

\begin{definition} \label{Virtual Singular Pure Braid Group}
We call the kernel of $\pi$ the \textit{virtual singular pure braid group}, and denote it by $VSPG_n$. We refer to an element in $VSPG_n$ as an \textit{extended virtual singular pure braid}. 
\end{definition}

Clearly, $VSPG_n$ is a normal subgroup of $VSG_n$ and $VSG_n/VSPG_n \cong S_n$. Our goal is to study the algebraic structure of the subgroup $VSPG_n$. 

\section{Another presentation for the virtual singular braid group} \label{another presentation}

From here on, all of the braids considered are extended virtual singular braids, thus elements of the group $VSG_n$. In this section we introduce a new presentation for $VSG_n$. This presentation uses as generators a special type of extended virtual singular pure braids, which we now define.   

 \begin{definition}\label{def:cstrings}
 The \textit{elementary fusing strings} $\mu_{i, i+1}^{\pm 1}, \, \gamma_{i, i+1},\, \bgamma_{i, i+1}$, where $1 \leq i \leq n-1$, are defined as follows:
   \begin{eqnarray} \label{FusingStrings}
   \mu_{i,i+1}: = \sigma_iv_i , \hspace{0.5cm} \mu_{i,i+1}^{-1}: = v_i \sigma_i^{-1} , \hspace{0.5cm} \gamma_{i, i+1}: = \tau_iv_i, \hspace{0.5cm} \bgamma_{i, i+1}: = v_i \btau_i .
   \end{eqnarray}
    \end{definition} 
  The elementary fusing strings are depicted in Figure~\ref{fig:FusingStrings}. Notice that these are pure braids, and that $\gamma_{i, i+1}^{-1} = \bgamma_{i,i+1}$.

\begin{figure}[ht] 
\[\mu_{i,i+1}  = \raisebox{-1.2cm} {\includegraphics[height=2.5cm]{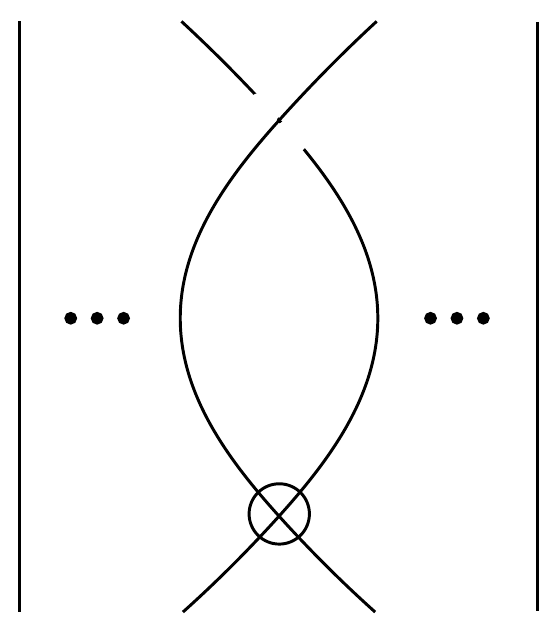}} \put(-48, 38){\fontsize{7}{7} $i \ \ \ \ i+1$} 
\hspace{1cm} \mu_{i,i+1}^{-1} = \raisebox{-1.2cm}{\includegraphics[height=2.5cm]{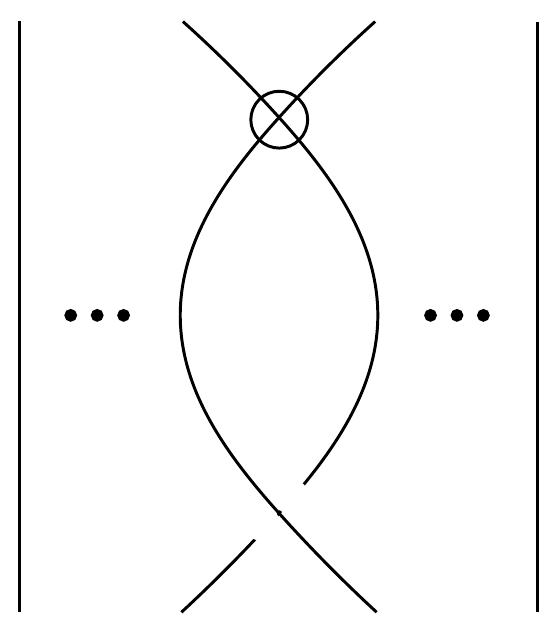}} \put(-48, 38){\fontsize{7}{7} $i \ \ i+1$}\]
\[ \gamma_{i,i+1}=  \raisebox{-1.2cm} {\includegraphics[height=2.5cm]{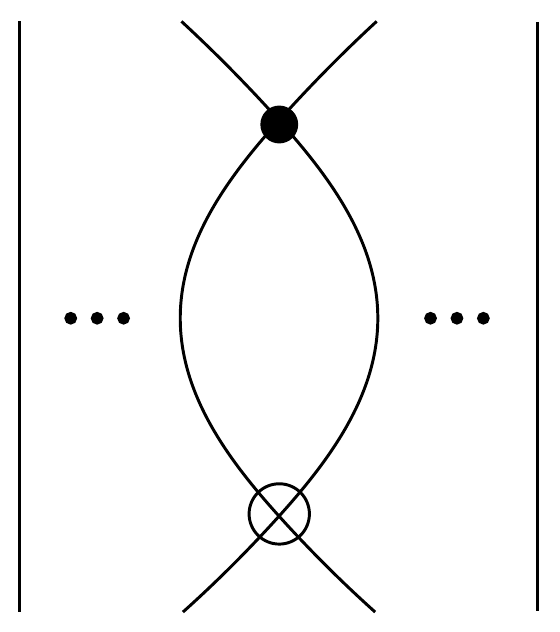}} \put(-48, 38){\fontsize{7}{7} $i \ \ \ \  i+1$} \
 \hspace{1cm} \bgamma_{i,i+1} = \raisebox{-1.2cm}{\includegraphics[height=2.5cm]{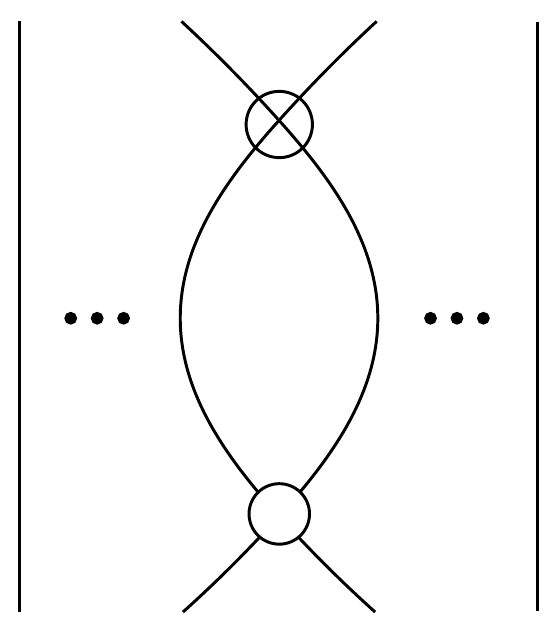}} \put(-48, 38){\fontsize{7}{7} $i \ \ i+1$} \]
\caption{The braids $\mu_{i, i+1},  \, \mu_{i, i+1}^{-1}, \gamma_{i, i+1}$ and $\bgamma_{i, i+1}$} 
\label{fig:FusingStrings}
\end{figure}

Elementary fusing strings are written as `interactions' between two adjacent strands $i$ and $i+1$ in a braid.  We can generalize the elementary fusing strings to involve any two non-adjacent strands $i$ and $j$ in a braid, as we now define.

\begin{definition} \label{Gammaij}
Let $1 \le i < j \le n-1$. The \textit{generalized fusing strings} $\mu_{ij}, \gamma_{ij}, \bgamma_{ij}$ and $\mu_{ji}, \gamma_{ji}, \bgamma_{ji}$ are elements of $VSPG_n$ defined as follows:
\[\mu_{ij}:=(v_{j-1}v_{j-2} \cdots v_{i+1})\mu_{i,i+1}(v_{i+1} \cdots v_{j-2}v_{j-1}),\]
\[\gamma_{ij}:=(v_{j-1}v_{j-2} \cdots v_{i+1})\gamma_{i,i+1}(v_{i+1} \cdots v_{j-2}v_{j-1}),\]
\[  \bgamma_{ij} := (v_{j-1}v_{j-2} \cdots v_{i+1})\bgamma_{i,i+1}(v_{i+1} \cdots v_{j-2}v_{j-1}), \]
\[\mu_{ji}:=(v_{j-1}v_{j-2} \cdots v_{i+1})v_i\mu_{i,i+1}v_i(v_{i+1} \cdots v_{j-2}v_{j-1}),\]
\[\gamma_{ji}:=(v_{j-1}v_{j-2} \cdots v_{i+1})v_i\gamma_{i,i+1}v_i(v_{i+1} \cdots v_{j-2}v_{j-1}),\]
\[  \bgamma_{ji} := (v_{j-1}v_{j-2} \cdots v_{i+1})v_i\bgamma_{i,i+1}v_i(v_{i+1} \cdots v_{j-2}v_{j-1}).  \]
\end{definition}

Note that $\bgamma_{ij}=\gamma_{ij}^{-1}$ and $\bgamma_{ji}=\gamma_{ji}^{-1}$. Moreover,
\[\mu_{ij}^{-1}=(v_{j-1}v_{j-2} \cdots v_{i+1})\mu_{i,i+1}^{-1}(v_{i+1} \cdots v_{j-2}v_{j-1}) \]
and
\[\mu_{ji}^{-1}=(v_{j-1}v_{j-2} \cdots v_{i+1})v_i\mu_{i,i+1}^{-1}v_i(v_{i+1} \cdots v_{j-2}v_{j-1}). \]
 Also, $\mu_{i+1,i}=v_i\mu_{i,i+1}v_i = v_i \sigma_i$ and $\gamma_{i+1,i}=v_i\gamma_{i,i+1}v_i=v_i\tau_i$. See Figure \ref{fig:GenFusingStrings} for the geometric representations of $\mu_{ij}, \mu_{ji}, \gamma_{ij}$ and $\gamma_{ji}$, where $i < j$.

\begin{figure}[ht]
\[\mu_{ij}= \raisebox{-2cm}{\includegraphics[height=4cm]{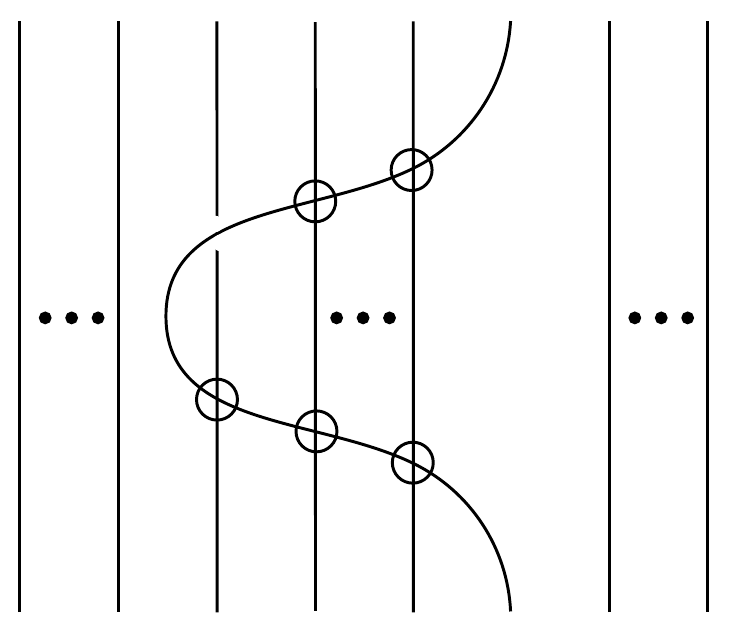}} \put(-98,60){\fontsize{7}{7} $i \hspace{1.8cm} j$} 
\hspace{1cm} \mu_{ji} = \raisebox{-2cm}{\includegraphics[height=4cm]{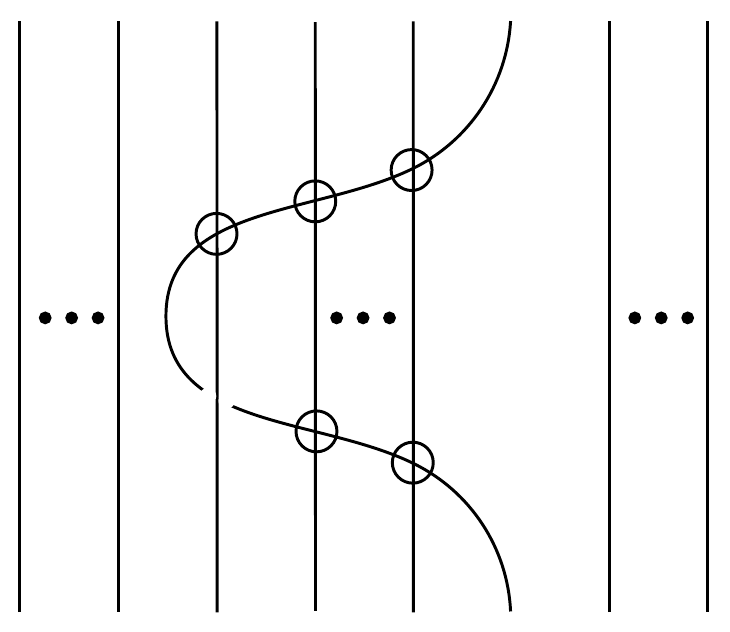}} \put(-98, 60){\fontsize{7}{7} $i \hspace{1.8cm} j$} \]

\[\gamma_{ij} = \raisebox{-2cm}{\includegraphics[height=4cm]{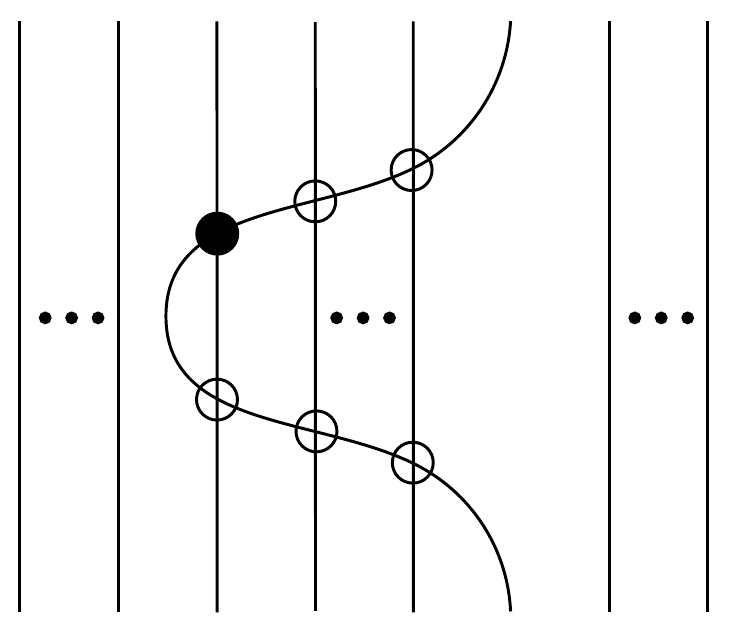} } \put(-98,60){\fontsize{7}{7} $i \hspace{1.8cm} j$} 
\hspace{1cm} \gamma_{ji} = \raisebox{-2cm}{\includegraphics[height=4cm]{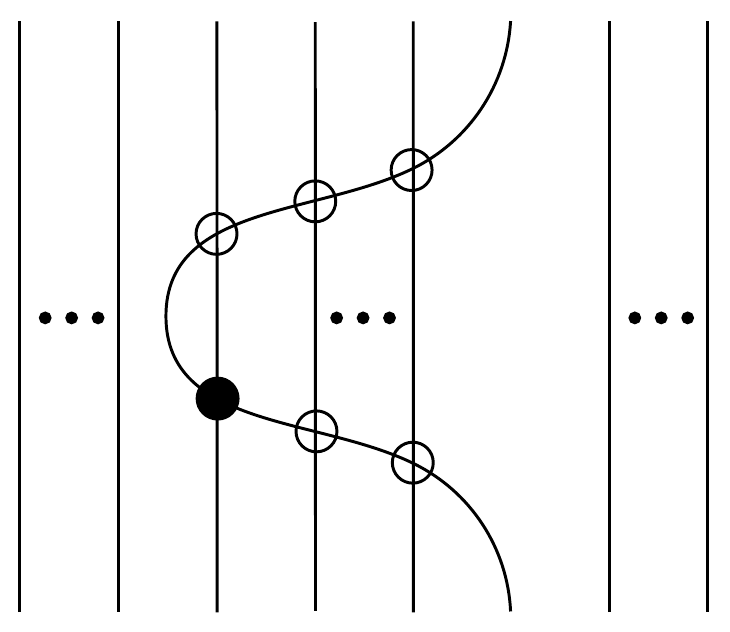}} \put(-98,60){\fontsize{7}{7} $i \hspace{1.8cm} j$}  \]
\caption{Braids representing $\mu_{ij}, \, \mu_{ji}, \, \gamma_{ij}$ and $\gamma_{ji}$, for $ i < j$ }
 \label{fig:GenFusingStrings}
\end{figure}

We remark that the elements $\mu_{ij}$ and $\mu_{ji}$ are called `connecting strings' in~\cite{KL2}, while $\mu_{ij}, \mu_{ji}$ and $\gamma_{ij}, \gamma_{ji}$ are called `fusing strings' in~\cite{Caprau}. We are using here the terminology from~\cite{Caprau}. 

\begin{lemma} \label{GammaBarRelations}
The following relations involving elementary fusing strings hold in $VSG_n$:
\begin{enumerate}
\item For all $1 \le i,j \le n-1$, $|i-j|=1$, \[v_i\mu_{j,j+1}v_i=v_j\mu_{i,i+1}v_j \quad \text{ and } \quad  v_i\gamma_{j,j+1}v_i=v_j\gamma_{i,i+1}v_j.\]
\item For all $1 \le i,j \le n-1$, $|i-j|=1$,
 \begin{eqnarray*}
 \mu_{j,j+1} (v_j\mu_{i,i+1} v_j) \mu_{i,i+1} &=& \mu_{i,i+1} (v_j\mu_{i,i+1}v_j) \mu_{j,j+1} \\
\mu_{j,j+1} (v_j\mu_{i,i+1}v_j) \gamma_{i,i+1} &=& \gamma_{i,i+1} (v_j\mu_{i,i+1}v_j) \mu_{j,j+1}.
\end{eqnarray*}
\item For all $1 \le i,j \le n-1$, $|i-j|>1$,
\[\alpha_i\beta_j=\beta_j\alpha_i, \quad \text{ where } \, \alpha_i,\beta_i \in \{\mu_{i,i+1},\gamma_{i,i+1}, v_i\}.\]
\item For all $1 \le i,j \le n-1$,
\[\mu_{i,i+1}v_i\gamma_{i,i+1}=\gamma_{i,i+1}v_i\mu_{i,i,+1}.\]
\end{enumerate}
\end{lemma}

\begin{proof}
It was proved in~\cite[Lemmas 5, 7, 8, 9]{Caprau} that all of these relations hold in the monoid $VSB_n$. (We warn the reader that the elements $\mu_{i, i+1}$ and $\gamma_{i, i+1}$ are denoted  $\mu_i$ and respectively $\gamma_i$ in~\cite{Caprau}). Therefore, these relations hold in the group $VSG_n$, as well. 
\end{proof}

We remark that there are other similar relations as those given in Lemma~\ref{GammaBarRelations}, but involving inverses of the elementary fusing strings: $\mu_{i, i+1}^{-1}$ and $\bgamma_{i, i+1}$. For example, we have that $v_i\mu_{j,j+1}^{-1}v_i=v_j\mu_{i,i+1}^{-1}v_j$ and $v_i\bgamma_{j,j+1}v_i=v_j\bgamma_{i,i+1}v_j$, for all $1 \le i,j \le n-1$, $|i-j|=1$. These relations are obtained by taking the inverses of the first two identities in Lemma~\ref{GammaBarRelations}. If we start with the fourth relation in Lemma~\ref{GammaBarRelations}, multiply both sides on the left and on the right by $\bgamma_{i,i+1}$ and use that $\gamma_{i, i+1}^{-1} = \bgamma_{i, i+1}$, we obtain $\bgamma_{i,i+1}\mu_{j,j+1} (v_j\mu_{i,i+1}v_j)  = (v_j\mu_{i,i+1}v_j) \mu_{j,j+1}\bgamma_{i,i+1}$, for all $1 \le i,j \le n-1$, $|i-j|=1$. Similarly, the last relation in the lemma above implies that $\bgamma_{i,i+1} \mu_{i,i+1} v_i =  v_i \mu_{i,i,+1} \bgamma_{i,i+1}$ holds, for all $1 \le i,j \le n-1$. Of course, there are other similar relations involving the elements $\mu_{i, i+1}^{-1}$.

We make use of Lemma~\ref{GammaBarRelations} to obtain a presentation for the virtual singular braid group $VSG_n$ using as generators the elementary fusing strings together with the virtual generators, as explained by the following theorem.

\begin{theorem} \label{VSGn Gamma Presentation}
The virtual singular braid group $VSG_n$ has an alternative presentation with generators $\{\mu_{i,i+1},\gamma_{i,i+1},v_i \mid 1 \le i \le n-1\}$ and subject to the following relations:
\begin{eqnarray}
v_i^2 &=& 1_n \label{eqn: G'n 1}\\
v_iv_jv_i &=& v_jv_iv_j,  \text{ for }  |i-j|=1
	\label{eqn: G'n 2} \\
 v_i\mu_{j,j+1}v_i &=& v_j\mu_{i,i+1}v_j, \text{ for }  |i-j|=1
         \label{eqn: G'n 3} \\
v_i\gamma_{j,j+1}v_i &=&v_j\gamma_{i,i+1}v_j,
	\text{ for }  |i-j|=1 \label{eqn: G'n 4} \\
 \mu_{j,j+1} (v_j\mu_{i,i+1}v_j) \mu_{i,i+1} &=& \mu_{i,i+1} (v_j\mu_{i,i+1}v_j) \mu_{j,j+1},
	\text{ for } |i-j|=1 \label{eqn: G'n 5} \\
\mu_{j,j+1} (v_j\mu_{i,i+1}v_j) \gamma_{i,i+1}&=& \gamma_{i,i+1} (v_j\mu_{i,i+1}v_j) \mu_{j,j+1},
	\text{ for } |i-j|=1 \label{eqn: G'n 6} \\
 \mu_{i,i+1}v_i\gamma_{i,i+1} &=& \gamma_{i,i+1}v_i\mu_{i,i+1} 
         \label{eqn: G'n 7} \\
\alpha_i\beta_j &=& \beta_j\alpha_i,
	\text{ for } |i-j|>1 \label{eqn: G'n 8}
\end{eqnarray}
where $\alpha_i,\beta_i \in \{ \mu_{i,i+1},\gamma_{i,i+1},v_i \}$.
\end{theorem}

\begin{proof}
It was proved in~\cite[Theorem 11]{Caprau} that the monoid $VSB_n$ has a presentation using as generators the elementary fusing strings $\mu_{i, i+1}^{\pm 1}, \gamma_{i, i+1}$ together with the virtual generators $v_i$, where $1 \le i \le n-1$, and subject to the same set of relations as those listed here in the statement of this theorem. Since the monoid $VSB_n$ embeds in the group with the same presentation, the statement of Theorem~\ref{VSGn Gamma Presentation} follows immediately. Alternatively, the theorem can be proved in the same manner as the analogous statement for $VSB_n$ was proved in~\cite{Caprau} .
\end{proof}

We remind the reader that we identify the generators $v_i$ with the transpositions $(i,i+1)$. The group $S_n$ acts on $VSG_n$ by conjugation, and we proceed now to study this action on the set of the generalized fusing strings.

\begin{lemma} \label{lemma:action}
Let $1 \le i < j \le n$. The following conjugation rules hold in $VSG_n$:
\begin{enumerate}
\item[(i)] $v_i\mu_{kl}v_i=\mu_{kl}$, \quad $v_i\gamma_{kl}v_i=\gamma_{kl}$, \quad for $|k-i|>1,|l-i|>1$
\item[(ii)] $v_{i-1}\mu_{i,i+1}v_{i-1}=\mu_{i-1,i+1}, \quad
	v_{i-1}\gamma_{i,i+1}v_{i-1}=\gamma_{i-1,i+1}$ \\ $v_{i-1}\mu_{i+1,i}v_{i-1}=\mu_{i+1,i-1}, \quad 
	v_{i-1}\gamma_{i+1,i}v_{i-1}=\gamma_{i+1,i-1}$
\item[(iii)] $v_i\mu_{i,i+1}v_i=\mu_{i+1,i}, \quad
	v_i\gamma_{i,i+1}v_i=\gamma_{i+1,i}$ \\
	$v_i\mu_{i+1,i}v_i=\mu_{i,i+1}, \quad v_i\gamma_{i+1,i}v_i=\gamma_{i,i+1}$
\item[(iv)] $v_{i+1}\mu_{i,i+1}v_{i+1}=\mu_{i,i+2}, \quad
	v_{i+1}\gamma_{i,i+1}v_{i+1}=\gamma_{i,i+2}$ \\ $v_{i+1}\mu_{i+1,i}v_{i+1}=\mu_{i+2,i}, \quad v_{i+1}\gamma_{i+1,i}v_{i+1}=\gamma_{i+2,i}$.
\end{enumerate}
There are similar relations with $\mu_{ij}$ and $\gamma_{ij}$ being replaced by $\mu_{ij}^{-1}$ and $\bgamma_{ij}$, respectively.
\end{lemma}

\begin{proof}
 Relations (i) follow directly from the commuting relations of $VSB_n$, and hence of $VSG_n$. Relations (ii) through (iv) follow from the definition of generalized fusing strings (recall Definition~\ref{Gammaij}). 
  \end{proof}

Due to the relations in Lemma~\ref{lemma:action}, the following corollary follows at once. 

\begin{corollary} \label{Sn Action-VS}
The group $S_n$ acts by conjugation on the set $\{ \mu_{ij}^{\pm 1}, \gamma_{ij}, \bgamma_{ij} \mid 1 \le i \ne j \le n \}$ by permuting the indices of the generalized fusing strings. That is,
\[\alpha \,  \mu_{ij} \, \alpha^{-1} = \mu_{\alpha(i) \alpha(j)}\,\,\, \text{and} \,\,\, \alpha \,  \gamma_{ij} \, \alpha^{-1} = \gamma_{\alpha(i) \alpha(j)}, \]
\[\alpha \,  \mu_{ij}^{-1} \, \alpha^{-1} = \mu_{\alpha(i) \alpha(j)}^{-1}\,\,\, \text{and} \,\,\, \alpha \,  \bgamma_{ij} \, \alpha^{-1} = \bgamma_{\alpha(i) \alpha(j)}, \]
where $\alpha \in S_n$. Moreover, this action is transitive.
\end{corollary}

As seen in Theorem~\ref{VSGn Gamma Presentation}, it is sufficient to work with the elementary fusing strings (together with the virtual generators) to generate the group $VSG_n$. This is not the case for $VSPG_n$, for which we need the generalized fusing strings.

\begin{lemma} \label{generatorsVSPGn}
The subgroup $VSPG_n$ of $VSG_n$ is generated by the elements  $\mu_{ij}, \gamma_{ij}$ for all $1 \le i \ne j \le n$.
\end{lemma} 

\begin{proof} Let $b \in VSG_n$. Then $b$ is written in terms of elementary virtual singular braids $\sigma_i^{\pm 1}, \tau_i$ and $\btau_i$.
By~\eqref{FusingStrings}, $\sigma_i = \mu_{i,i+1} v_i, \,\, \sigma_i^{-1} = v_i \mu_{i, i+1}^{-1}, \,\, \tau_i = \gamma_{i,i+1} v_i$ and $\btau_i =  v_i \bgamma_{i, i+1}$. By Corollary~\ref{Sn Action-VS}, 
\[\alpha \mu_{ij}^{\pm 1} = \mu_{\alpha(i) \alpha(j)} ^{\pm 1}\alpha \,\, , \,\,  \alpha \,  \gamma_{ij} = \gamma_{\alpha(i) \alpha(j)} \alpha \,\, \text{ and }  \alpha \,  \bgamma_{ij} = \bgamma_{\alpha(i) \alpha(j)} \alpha,  \]
for every $\alpha \in \{v_k \, | \, 1 \leq k \leq n-1 \}$.
Hence $\mu_{ij}^{\pm 1}, \gamma_{ij}$ and $\bgamma_{ij}$ can move on top of the braid (equivalently, in front of the word representing the braid). Therefore, $b = a c$, where $a$ is a word over the alphabet $\{\mu_{ij}, \mu_{ij}^{-1}, \gamma_{ij}, \bgamma_{ij},  | \, 1 \le i \ne j \le n\}$ and $c$ is a word over the alphabet $\{ v_k \, | \, 1 \leq k \leq n-1 \}$. Thus $a \in VSPG_n$ and $c \in S_n$. It follows that $VSG_n = (VSPG_n) \cdot S_n$.

If $b \in VSPG_n$, then $1_n = \pi(b) =  \pi(a) \pi(c) = 1_n \pi(c) = \pi(c)$. Hence, $c$ must be the identity permutation in $S_n$ and $b = a$. The statement follows.
 \end{proof}
 
\begin{corollary}
$VSG_n \cong VSPG_n  \rtimes_{\phi} S_n$, where $\phi \co S_n \to \text{Aut}(VSPG_n)$ is the permutation representation associated with the action of $S_n$ on $VSPG_n$, as described in Corollary~\ref{Sn Action-VS}.
\end{corollary}

\begin{proof}
 We know that $VSPG_n$ is a normal subgroup of $VSG_n$. From the proof of Lemma~\ref{generatorsVSPGn}, we have that $VSG_n = (VSPG_n) \cdot S_n$ and $VSPG_n \cap S_n = \{1_n\}$. The result follows from a well-known theorem  that recognizes groups as (internal) semi-direct products.
 
 We can also use short exact sequences to prove this. Note that 
 \[ 1 \longrightarrow VSPG_n \stackrel{i}{\longrightarrow} VSG_n \stackrel{\pi}{\longrightarrow} S_n \longrightarrow 1  \]
 is a split short exact sequence with splitting (section) $s \co S_n \to VSG_n$ that sends the transpositions $(i, i+1)$ to $v_i$, for all $1 \leq i \leq n-1$. Hence, $VSG_n \cong VSPG_n  \rtimes S_n$. 
 \end{proof}


\section{Generators and relations of the virtual singular pure braid group} \label{presentation VSPGn}

In this section, we provide a presentation via generators and relations for the virtual singular pure braid group, $VSPG_n$. For this purpose, we use the Reidemeister-Schreier method (see~\cite[Chapter 2.2]{MKS}). We note that Bardakov~\cite[Theorem 1]{Bardakov} used this method to provide a presentation for the virtual pure braid group, $VP_n$. Similarly, Bardakov and Bellingeri~\cite[Proposition 17]{BB} made use of this method to find a presentation for the normal closure of the braid group $B_n$ in the virtual braid group $VB_n$. Moreover, a similar proof was provided by the first author and Zepeda in~\cite[Theorem 17]{Caprau}, to construct a presentation for the virtual singular pure braid monoid $VSP_n$.

\begin{theorem} \label{VSPGn Gamma Presentation} 
The virtual singular pure braid group $VSPG_n$ is generated by the generalized fusing strings $\{\mu_{ij}, \gamma_{ij} \ |\ 1 \le i \ne j \le n \}$, subject to the following relations:
\begin{eqnarray}
\mu_{ij}\mu_{ik}\mu_{jk} &=&\mu_{jk}\mu_{ik}\mu_{ij} \label{eq:pureYB}\\
\mu_{ij}\mu_{ik}\gamma_{jk} &=&\gamma_{jk}\mu_{ik}\mu_{ij} \label{eq:pureYB-mixt}\\
 \gamma_{ij}\mu_{ik}\mu_{jk} &=&\mu_{jk}\mu_{ik}\gamma_{ij}  \label{eq:pureYB-mixt-two} \\
\mu_{ij}\gamma_{ji} &=& \gamma_{ij}\mu_{ji} \label{eq:pureTwist}  \\
\mu_{ij} \mu_{kl} = \mu_{kl} \mu_{ij} \, ,\,\,\,  \,\,\,\, \gamma_{ij} \gamma_{kl} &=& \gamma_{kl} \gamma_{ij}  \, ,\,\,\, \,\,\,\, \mu_{ij} \gamma_{kl} = \gamma_{kl} \mu_{ij} \,, \label{eq:purecommuting}
\end{eqnarray}
where distinct letters stand for distinct indices.
\end{theorem}

\begin{proof}
The proof and notation is similar to that for the virtual singular pure braid monoid in~\cite[Theorem 17]{Caprau}. 

By Lemma~ \ref{generatorsVSPGn}, we know that $VSPG_n$ is generated by $\{\mu_{ij}, \gamma_{ij} \ |\ 1 \le i \ne j \le n \}$. The Reidemeister-Schreier method will allow us to find the corresponding relations for this group, and along the steps, we will verify again that $\mu_{ij}, \gamma_{ij}$, for all $1 \le i \ne j \le n$, generate the normal subgroup $VSPG_n$.
We start by considering a Schreier system of right coset representatives of the subgroup $VSPG_n$ in $VSG_n$:
\begin{eqnarray*}
\Lambda_n : &=& \{ (v_{i_1}v_{i_1 -1} \cdots v_{i_1 - r_1} ) (v_{i_2}v_{i_2 -1} \cdots v_{i_2 - r_2} ) \dots (v_{i_p}v_{i_p -1} \cdots v_{i_p - r_p} ) |  \\
&& \hspace{3cm} 1 \leq i_1 < i_2 < \dots < i_p \leq n-1; 0 \leq r_j < i_j      \}.
\end{eqnarray*}
A Schreier right coset representative has the property that any initial segment of a representative is also a coset representative. Using~\cite[Chapter 2.2]{MKS}, we define the map $\bar{\empty} : VSG_n \to \Lambda_n$ which takes an extended virtual singular braid $\omega$ in $VSG_n$ to a coset representative $\overline{\omega}$ in $\Lambda_n$. Then $\omega \in (VSPG_n) \, \overline{\omega}$ and so $\omega = \alpha \, \overline{\omega}$, for some $\alpha \in VSPG_n$. Hence $\omega \, \overline{\omega}^{-1} \in VSPG_n$.

By~\cite[Theorem 2.7]{MKS}, the subgroup $VSPG_n$ of $VSG_n$  is generated by all elements of the form
$s_{\lambda, a} = \lambda a \, (\overline{\lambda a})^{-1},$ where $\lambda \in \Lambda_n$ and $a$ is a generator for the $VSG_n$. Recall that the generators for $VSG_n$ are $v_i, \sigma_i$, and $\tau_i$, for all $1 \leq i \leq n-1$.
We note that $\overline{\lambda v_i} = \lambda v_i, \, \overline{\lambda \sigma_i}  = \lambda v_i$ and $\overline{\lambda \tau_i} = \lambda v_i $, for all $\lambda \in \Lambda_n$ and $1 \leq i \leq n-1$. 
It follows that for all $\lambda \in \Lambda_n$ and $1 \leq i \leq n-1$, we have:
 \begin{eqnarray*}
 s_{\lambda, v_i} &=& 1_n, \\
  s_{\lambda, \sigma_i} & =& \lambda (\sigma_i v_i)     \lambda^{-1} = \lambda \mu_{i, i+1} \lambda^{-1}, \\
  s_{\lambda, \tau_i} &=& \lambda (\tau_i v_i) \lambda^{-1} = \lambda \gamma_{i, i+1} \lambda^{-1}.
  \end{eqnarray*}

Since the elements in $\Lambda_n$ are identified with elements of $S_n$, and since by Corollary~\ref{Sn Action-VS}, $S_n$ acts on the elementary fusing strings $\mu_{i, i+1}$ and $\gamma_{i, i+1}$ by permuting the indices, the above calculations imply that 
each generator $s_{\lambda, \sigma_i}$ is equal to some $\mu_{kl}$ and that each $s_{\lambda, \tau_i}$ is equal to some $\gamma_{kl}$. Conversely, $\mu_{kl}$ and $\gamma_{kl}$ are equal to some $s_{\lambda, \sigma_i}$ and $s_{\lambda, \tau_i}$, respectively. It follows that $VSPG_n$ is generated by $\{\mu_{ij}, \gamma_{ij} \ |\ 1 \le i \ne j \le n \}$.

We continue with the Reidemeister-Schreier method to find the defining relations for the group $VSPG_n$ and thus prove the second part of the theorem. Following~\cite[Theorem 2.7]{MKS}, we define a rewriting process $\mathcal{R}$ which converts an element in $VSPG_n$ written in terms of the generators of the parent group $VSG_n$ into a word written in terms of the generators of the subgroup $VSPG_n$. 

If $ \omega = a_1^{\epsilon_1} a_2^{\epsilon_2} \cdots a_t^{\epsilon_t} \in VSPG_n$, where $a_j \in \{\sigma_j,  \tau_j, v_j  \, | \, 1 \leq j \leq n-1 \},$ and $\epsilon_j = \pm 1$ (recall that $\tau_j^{-1} = \overline{\tau}_j$)
 then the rewritten word $\mathcal{R}(\omega)$ is given by
 \[ \mathcal{R}(\omega) = s_{k_1, a_1}^{\epsilon_1} s_{k_2, a_2}^{\epsilon_2} \cdots s_{k_t, a_t}^{\epsilon_t},  \]
 where for each $1 \leq j \leq t$, $k_j = \begin{cases} \overline{a_1 a_2 \cdots a_{j-1}} \,\, \,\, \text{if} \,\, \epsilon_j = 1,\\ \overline{a_1 a_2 \cdots a_{j-1} a_j} \,\, \,\, \text{if} \,\, \epsilon_j = -1. \end{cases}$

In this proof, it is convenient to work with the standard presentation for the virtual singular braid group $VSPG_n$; for convenience, we list the relations here again:

\begin{itemize}
\item two-point relations: $v_i^2 = 1_n$, $\sigma_i \tau_i = \tau_i \sigma_i$.
\item three-point relations:
\begin{eqnarray*}
\sigma_i\sigma_{i+1}\sigma_i=\sigma_{i+1}\sigma_i\sigma_{i+1} \hspace{1cm} v_i v_{i+1} v_i= v_{i+1} v_i v_{i+1},\\
 v_i \sigma_{i+1} v_i= v_{i+1} \sigma_i v_{i+1} \hspace{1.1cm} v_i \tau_{i+1} v_i= v_{i+1} \tau_i v_{i+1}, \\
 \sigma_i\sigma_{i+1}\tau_i=\tau_{i+1}\sigma_i\sigma_{i+1} \hspace{1cm} \sigma_{i+1}\sigma_i\tau_{i+1}=\tau_i\sigma_{i+1}\sigma_i.
 \end{eqnarray*}
\item commuting relations: $g_ih_j = h_jg_i $ for $|i-j| >1$, where $g_i, h_i \in \{ \sigma_i, \tau_i, v_i  \}$.
\end{itemize}

The relations for the subgroup $VSPG_n$ are given by $\mathcal{R}(\lambda r_u \lambda^{-1})$, for all $\lambda \in \Lambda_n$ and all defining relations $r_u$ of the parent group $VSG_n$. 

Starting with the relation $r_1: v_i^2 = 1_n$ and noting that $\mathcal{R}(1_n) = 1_n$ and $s_{\lambda, v_i} = 1_n$ for any $\lambda \in \Lambda_n$, it is clear that the rewriting process $\mathcal{R}$ applied to the relation $r_1$ produces the trivial relation $1_n = 1_n$.
By applying the rewriting process to $r_2: \sigma_i \tau_i = \tau_i \sigma_i$, we get the following:
\begin{eqnarray*}
\mathcal{R}(\sigma_i \tau_i ) &=& s_{1_n, \sigma_i} \, s_{\overline{\sigma_i}, \tau_i}  = s_{1_n, \sigma_i} \, s_{v_i, \tau_i} = (\mu_{i,i+1}) (v_i \gamma_{i, i+1}v_i) = \mu_{i,i+1} \, \gamma_{i+1, i}\\
\mathcal{R}(\tau_i \sigma_i  ) &=& s_{1_n, \tau_i} \, s_{\overline{\tau_i}, \sigma_i}  = s_{1_n, \tau_i} \, s_{v_i, \sigma_i} = (\gamma_{i,i+1}) (v_i \mu_{i, i+1}v_i) = \gamma_{i,i+1} \, \mu_{i+1, i}.
\end{eqnarray*}
Therefore, $\mathcal{R}(r_2)$ has the form $\mu_{i,i+1} \, \gamma_{i+1, i} = \gamma_{i,i+1} \, \mu_{i+1, i}$.  Conjugating by all representatives in $\Lambda_n$, it follows that relations $\mathcal{R}(\lambda r_2 \lambda^{-1})$ have the form $\mu_{kl} \gamma_{lk} = \gamma_{kl} \mu_{lk}$, where $k \not = l$.

The rewriting process applied to relation $r_3: \sigma_i \sigma_{i+1} \sigma_i = \sigma_{i+1} \sigma_i \sigma_{i+1}$ yields:
\[
\mathcal{R} (\sigma_i \sigma_{i+1} \sigma_i ) = s_{1_n, \sigma_i} \, s_{\overline{\sigma_i}, \sigma_{i+1}} \, s_{\overline{\sigma_i \sigma_{i+1} }, \sigma_i} = s_{1_n, \sigma_i} \, s_{v_i, \sigma_{i+1}} \, s_{v_i v_{i+1, \sigma_i}}.
\]
Since, 
\begin{eqnarray*}
s_{1_n, \sigma_i} &=& \mu_{i,i+1}\\
s_{v_i, \sigma_{i+1}} &=& v_i \mu_{i+1, i+2} v_i = \mu_{i, i+2}\\
s_{v_i v_{i+1, \sigma_i}} &=& v_i v_{i+1}\mu_{i,i+1} v_{i+1}v_i = \mu_{i+1, i+2}, 
\end{eqnarray*}
we obtain that $\mathcal{R} (\sigma_i \sigma_{i+1} \sigma_i ) = \mu_{i, i+1} \, \mu_{i, i+2} \, \mu_{i+1, i+2}$.
On the other hand,
\begin{eqnarray*}
\mathcal{R} (\sigma_{i+1} \sigma_i \sigma_{i+1} ) &=& s_{1_n, \sigma_{i+1}} \, s_{\overline{\sigma_{i+1}}, \sigma_i} \, s_{\overline{\sigma_{i+1} \sigma_i}, \sigma_{i+1} }
= s_{1_n, \sigma_{i+1}} \, s_{v_{i+1}, \sigma_i} s_{v_{i+1} v_i, \sigma_{i+1}}  \\
&=& \mu_{i+1, i+2} \, \mu_{i, i+2} \, \mu_{i, i+1}.
\end{eqnarray*}
We obtained that the rewriting process applied to the relation $r_3: \sigma_i \sigma_{i+1} \sigma_i = \sigma_{i+1} \sigma_i \sigma_{i+1}$ yields $\mu_{i, i+1} \, \mu_{i, i+2} \, \mu_{i+1, i+2} = \mu_{i+1, i+2} \, \mu_{i, i+2} \, \mu_{i, i+1}.$ Furthermore, conjugating this identity by all representatives in $\Lambda_n$ results in the relations $\mu_{ij}\mu_{ik}\mu_{jk} =\mu_{jk}\mu_{ik}\mu_{ij}$.

We consider now the relation $r_4: v_iv_{i+1} v_i = v_{i+1} v_i v_{i+1}$. Since for any $\lambda \in \Lambda_n$, $s_{\lambda, v_i} = 1_n$, it is easy to see that $\mathcal{R} (r_4)$ yields the trivial relation $1_n = 1_n$ in $VSPG_n$. Similarly, relations $r_5: v_i \sigma_{i+1} v_i = v_{i+1} \sigma_i v_{i+1}$ and $r_6: v_i \tau_{i+1} v_i = v_{i+1} \tau_i v_{i+1}$ induce the trivial relation. We show this for relation $r_5$ only, to avoid repetition.
\begin{eqnarray*}
\mathcal{R}(v_i \sigma_{i+1} v_i) &=& s_{1_n, v_i} \, s_{\overline{v_i}, \sigma_{i+1}} \, s_{\overline{v_i \sigma_{i+1}}, v_i} = 1_n  \,s_{v_i, \sigma_{i+1}} \, 1_n = v_i \mu_{i+1, i+2}v_i = \mu_{i, i+2}\\
\mathcal{R}(v_{i+1} \sigma_i v_{i+1}) &=& s_{1_n, v_{i+1}} \, s_{\overline{v_{i+1}}, \sigma_{i}} \, s_{\overline{v_{i+1} \sigma_{i}}, v_{i+1}} = 1_n  \,s_{v_{i+1}, \sigma_{i}} \, 1_n = v_{i+1} \mu_{i, i+1}v_{i+1} = \mu_{i, i+2}.
\end{eqnarray*}

Working with relation $r_7: \sigma_i \sigma_{i+1} \tau_i = \tau_{i+1} \sigma_i \sigma_{i+1}$, we obtain:
\begin{eqnarray*}
\mathcal{R}(\sigma_i \sigma_{i+1} \tau_i ) &=& s_{1_n, \sigma_i} \, s_{v_i, \sigma_{i+1}} \, s_{v_i v_{i+1, \tau_i}} = \mu_{i, i+1} \, \mu_{i, i+2} \, \gamma_{i+1, i+2}\\
\mathcal{R} (\tau_{i+1} \sigma_i \sigma_{i+1} ) &=&s_{1_n, \tau_{i+1}} \, s_{v_{i+1}, \sigma_i} s_{v_{i+1} v_i, \sigma_{i+1}} = \gamma_{i+1, i+2} \, \mu_{i, i+2} \, \mu_{i, i+1}.
\end{eqnarray*}
It follows that the rewriting process $\mathcal{R}$ applied to the relation $r_7$ in the group $VSG_n$ yields the generating relation $\mu_{i, i+1} \, \mu_{i, i+2} \, \gamma_{i+1, i+2} = \gamma_{i+1, i+2} \, \mu_{i, i+2} \, \mu_{i, i+1}$ for the normal subgroup $VSPG_n$. Conjugating the latter relation by all representatives $\lambda$ in $\Lambda_n$, it follows that the relations $\mathcal{R}(\lambda r_7 \lambda^{-1})$ have the form $\mu_{ij}\mu_{ik}\gamma_{jk} =\gamma_{jk}\mu_{ik}\mu_{ij}$.

We apply now the rewriting process $\mathcal{R}$ to the relation $r_8: \sigma_{i+1}\sigma_i \tau_{i+1} = \tau_i \sigma_{i+1}\sigma_i$.
\begin{eqnarray*}
\mathcal{R}(\sigma_{i+1}\sigma_i \tau_{i+1}) &=& s_{1_n, \sigma_{i+1}} \, s_{v_{i+1}, \sigma_i } \, s_{v_{i+1} v_i, \tau_{i+1}} = \mu_{i+1, i+2} \, \mu_{i, i+2} \, \gamma_{i, i+1}\\
\mathcal{R} (\tau_i \sigma_{i+1}\sigma_i) &=&s_{1_n, \tau_i} \, s_{v_i, \sigma_{i+1}} s_{v_i v_{i+1}, \sigma_i} =\gamma_{i, i+1} \, \mu_{i, i+2} \, \mu_{i+1, i+2} .
\end{eqnarray*}
Therefore, $\gamma_{i, i+1} \, \mu_{i, i+2} \, \mu_{i+1, i+2} = \mu_{i+1, i+2} \, \mu_{i, i+2} \, \gamma_{i, i+1}$ is a generating relation for the normal subgroup $VSPG_n$. Conjugating the latter relation by all representatives in $\Lambda_n$, we get relations $ \gamma_{ij}\mu_{ik}\mu_{jk} = \mu_{jk}\mu_{ik}\gamma_{ij}$.

We are left with showing that the rewriting process applied to the commuting relations of $VSG_n$ yield the commuting relations~\eqref{eq:purecommuting} in $VSPG_n$. We consider the relation $\tau_i \tau_j = \tau_j \tau_i$, where $|i-j| >1$, and apply the rewriting process to it to derive the relation $\gamma_{i, i+1} \gamma_{j, j+1} = \gamma_{j, j+1} \gamma_{i, i+1}$, as we show below.
\begin{eqnarray*}
\mathcal{R}(\tau_i \tau_j ) &=&s_{1_n, \tau_i} \, s_{v_i, \tau_j}  = (\gamma_{i, i+1}) (v_i \gamma_{j, j+1}v_i) = \gamma_{i,i+1} \, \gamma_{j, j+1}\\
\mathcal{R}( \tau_j \tau_i) &=&s_{1_n, \tau_j} \, s_{v_j, \tau_i}  = (\gamma_{j, j+1}) (v_j \gamma_{i, i+1}v_j) =  \gamma_{j, j+1} \, \gamma_{i,i+1}. 
\end{eqnarray*}
The last equalities in both lines above above hold because $|i-j| >1$, and thus $v_i \gamma_{j, j+1}v_i = \gamma_{j, j+1}$ and $v_j \gamma_{i, i+1}v_j =  \gamma_{i,i+1}$. Conjugating the relation $\gamma_{i, i+1} \gamma_{j, j+1} = \gamma_{j, j+1} \gamma_{i, i+1}$ by all representatives in $\Lambda_n$, we get:
\[ \gamma_{ij} \gamma_{kl} = \gamma_{kl} \gamma_{ij}, \,\, \text{where}\,\, \{i, j\} \cap \{ k, l\} = \emptyset. \]
Similarly, by applying the rewriting process to relations $\sigma_i \sigma_j = \sigma_j \sigma_i$ and $\sigma_i \tau_j = \tau_j \sigma_i$, where $|i-j| >1$, we obtain:
\[ \mu_{ij} \mu_{kl} = \mu_{kl} \mu_{ij} \,\, \text{and} \,\,\mu_{ij} \gamma_{kl} = \gamma_{kl} \mu_{ij}, \,\, \text{where}\,\, \{i, j\} \cap \{ k, l\} = \emptyset.  \] 
 The commuting relations $\sigma_i v_j = v_j \sigma_i$ and $\tau_i v_j = v_j \tau_i$ of the parent group $VSG_n$ result in the identity relation in $VSPG_n$.
 
 We have applied the rewriting process $\mathcal{R}$ to all of the defining relations for the group $VSG_n$ and we arrived at the relations~\eqref{eq:pureYB} through~\eqref{eq:purecommuting}. Hence, the second part of the theorem holds.
\end{proof}

We note that there are other similar relations as those in Theorem~\ref{VSPGn Gamma Presentation} that hold among the generalized fusing strings. For example, the second relations from Theorem~\ref{VSPGn Gamma Presentation} yield the relation $\mu_{ik} \gamma_{jk}\mu_{ij}^{-1} = \mu_{ij}^{-1} \gamma_{jk} \mu_{ik}$.
Also, the fourth relations $\mu_{ij}\gamma_{ji}=\gamma_{ij}\mu_{ji}$  imply that $\bgamma_{ij} \mu_{ij} = \mu_{ji} \bgamma_{ji}$ and $\gamma_{ji} \mu_{ji}^{-1} =  \mu_{ij}^{-1} \gamma_{ij}$, for all $1 \le i \ne j \le n$.


\section{The virtual singular pure braid group as a semi direct product} \label{semi-direct product}

The scope of this section is to represent the virtual singular pure braid group $VSPG_n$ as a semi-direct product of $n-1$ subgroups and study the structures of these subgroups. We will also show how these results yield normal form of words in $VSPG_n$ and $VSG_n$.

From here on, we will denote the conjugation $b^{-1}ab$ by $a^b$. Moreover, $a^{-b}=(a^{-1})^b$.

\begin{lemma} \label{VSPGn Conjugation}
The following equalities hold in the group $VSPG_n$:
\begin{enumerate}
\item[(i)] for distinct $i,j,k,l$, $\max\{i,j\}<\max\{k,l\}$, 
\[\alpha^{\beta^{\pm 1}}=\alpha, \text{ where } \alpha \in \{ \mu_{kl}, \gamma_{kl}\}, \beta \in \{ \mu_{ij},\gamma_{ij} \};\]


\item[(ii)] for $i<j<k$ or $j<i<k$,
\[\mu_{ik}^{\mu_{ij}}=\mu_{kj}^{\mu_{ij}}\mu_{ik}
\mu_{kj}^{-1}, \quad \mu_{ik}^{\mu_{ij}^{-1}}
=\mu_{kj}^{-1}\mu_{ik}\mu_{kj}^{\mu_{ij}^{-1}},\]


\[\mu_{ik}^{\gamma_{ij}}= \gamma_{jk}^{-\gamma_{ij}} \mu_{kj}^{-\gamma_{ij}} \gamma_{kj}^{\gamma_{ij}} \mu_{ik} \mu_{jk},             
 \quad \mu_{ik}^{\bgamma_{ij}} =  \mu_{jk} \mu_{ik}\gamma_{jk}^{-\bgamma_{ij}} \mu_{kj}^{-\bgamma_{ij}} \gamma_{kj}^{\bgamma_{ij}} ,\]

\[\gamma_{ik}^{\mu_{ij}}=\mu_{kj}^{\mu_{ij}}\gamma_{ik}
\mu_{kj}^{-1}, \quad \gamma_{ik}^{\mu_{ij}^{-1}}
=\mu_{kj}^{-1}\gamma_{ik}\mu_{kj}^{\mu_{ij}^{-1}};\]

\item[(iii)] for $i<j<k$ or $j<i<k$,
\[\mu_{ki}^{\mu_{ij}}=\mu_{kj}\mu_{ki}\mu_{kj}^{-\mu_{ij}}, \quad \mu_{ki}^{\mu_{ij}^{-1}}=\mu_{kj}^{-\mu_{ij}^{-1}}\mu_{ki}\mu_{kj},\]
\[\mu_{ki}^{\gamma_{ij}}=\mu_{kj}\mu_{ki}\mu_{kj}^{-\gamma_{ij}}, \quad \mu_{ki}^{\bgamma_{ij}}=\mu_{kj}^{-\bgamma_{ij}}\mu_{ki} \mu_{kj},\]
\[\gamma_{ki}^{\mu_{ij}}=\mu_{kj}\gamma_{ki}\mu_{kj}^{-\mu_{ij}}, \quad \gamma_{ki}^{\mu_{ij}^{-1}}=\mu_{kj}^{-\mu_{ij}^{-1}}\gamma_{ki}\mu_{kj};\]


\item[(iv)] for $i<j<k$ or $j<i<k$, \[\mu_{jk}^{\mu_{ij}}=\mu_{ik}\mu_{jk}\mu_{kj}
\mu_{ik}^{-1}\mu_{kj}^{-\mu_{ij}}, \quad \mu_{jk}^{\mu_{ij}^{-1}}=\mu_{kj}^{-\mu_{ij}^{-1}}\mu_{ik}^{-1}\mu_{kj}\mu_{jk}\mu_{ik},\]


\[\mu_{jk}^{\gamma_{ij}} = \gamma_{jk}^{\gamma_{ij}} \mu_{kj}^{\gamma_{ij}} \gamma_{kj}^{-\gamma_{ij}},
\quad \mu_{jk}^{\bgamma_{ij}} = \gamma_{jk}^{\bgamma_{ij}} \mu_{kj}^{\bgamma_{ij}} \gamma_{kj}^{-\bgamma_{ij}}, \]


\[\gamma_{jk}^{\mu_{ij}}=\mu_{ik}\gamma_{jk}\mu_{kj}
\mu_{ik}^{-1}\mu_{kj}^{-\mu_{ij}}, \quad \gamma_{jk}^{\mu_{ij}^{-1}}=\mu_{kj}^{-\mu_{ij}^{-1}}\mu_{ik}^{-1}\mu_{kj}\gamma_{jk}\mu_{ik};\]


\item[(v)] for $i<j<k$ or $j<i<k$,
\[\gamma_{kj}^{\mu_{ij}}=\mu_{kj}^{\mu_{ij}}\mu_{ik}\mu_{kj}^{-1}\gamma_{kj}\mu_{ik}^{-1} \ \text{ and } \ \gamma_{kj}^{\mu_{ij}^{-1}}=\mu_{ik}^{-1}\gamma_{kj}\mu_{kj}^{-1}\mu_{ik}\mu_{kj}^{\mu_{ij}^{-1}}. \]
\end{enumerate}
\end{lemma}

\begin{proof} Let $1\le i,j,k,l \le n$, where different letters stand for distinct indices. 

(i) This statement follows immediately from the commuting relations~\eqref{eq:purecommuting} given in Theorem \ref{VSPGn Gamma Presentation}. 

(ii) The first two identities follow from the relation $\mu_{ik}\mu_{ij}\mu_{kj}=\mu_{kj}\mu_{ij}\mu_{ik}$ (this is relation~\eqref{eq:pureYB}), as demonstrated below:
\[\mu_{ik}^{\mu_{ij}} = \mu_{ij}^{-1}\mu_{ik}\mu_{ij} = \mu_{ij}^{-1} (\mu_{ik}\mu_{ij} \mu_{kj})\mu_{kj}^{-1} = (\mu_{ij}^{-1}\mu_{kj}\mu_{ij})\mu_{ik}\mu_{kj}^{-1}= \mu_{kj}^{\mu_{ij}}\mu_{ik}\mu_{kj}^{-1};   \]
\[\mu_{ik}^{\mu_{ij}^{-1}} = \mu_{ij}\mu_{ik}\mu_{ij}^{-1} = \mu_{kj}^{-1} (\mu_{kj} \mu_{ij}\mu_{ik})\mu_{ij}^{-1} =  \mu_{kj}^{-1}\mu_{ik}(\mu_{ij}\mu_{kj}\mu_{ij}^{-1})= \mu_{kj}^{-1}\mu_{ik}\mu_{kj}^{\mu_{ij}^{-1}}. \]


The third and fourth identities in (ii) follow from the relations $\gamma_{ij}\mu_{ik}\mu_{jk}=\mu_{jk}\mu_{ik}\gamma_{ij}$ and $\gamma_{jk} \mu_{jk}^{-1} = \mu_{kj}^{-1} \gamma_{kj}$ (see relations~\eqref{eq:pureYB-mixt-two} and~\eqref{eq:pureTwist}), as we now show:

\[\mu_{ik}^{\gamma_{ij}}= \bgamma_{ij}\mu_{ik}\gamma_{ij} = \bgamma_{ij} \mu_{jk}^{-1} (\mu_{jk}  \mu_{ik}\gamma_{ij}) 
 =  \bgamma_{ij} \mu_{jk}^{-1} \gamma_{ij} \mu_{ik} \mu_{jk} = (\mu_{jk}^{-1})^{\gamma_{ij}}\mu_{ik} \mu_{jk}. \]
 But since $\gamma_{jk} \mu_{jk}^{-1} = \mu_{kj}^{-1} \gamma_{kj}$, we have that $\mu^{-1}_{jk} = \bgamma_{jk} \mu_{kj}^{-1} \gamma_{kj}$, and thus
 
 \[\mu_{ik}^{\gamma_{ij}}= (\mu_{jk}^{-1})^{\gamma_{ij}}\mu_{ik} \mu_{jk} = (\bgamma_{jk} \mu_{kj}^{-1} \gamma_{kj})^{\gamma_{ij}} \mu_{ik} \mu_{jk} =
  \gamma_{jk}^{-\gamma_{ij}} \mu_{kj}^{-\gamma_{ij}} \gamma_{kj}^{\gamma_{ij}} \mu_{ik} \mu_{jk}.  \]
  Similarly,
 
\[\mu_{ik}^{\bgamma_{ij}} =\gamma_{ij}\mu_{ik}\bgamma_{ij} = (\gamma_{ij}\mu_{ik} \mu_{jk}) \mu_{jk}^{-1} \bgamma_{ij} = \mu_{jk} \mu_{ik} \gamma_{ij}\mu_{jk}^{-1} \bgamma_{ij}  = \mu_{jk} \mu_{ik} \mu_{jk}^{-\bgamma_{ij}}.
 \]
  Using again that $\mu^{-1}_{jk} = \bgamma_{jk} \mu_{kj}^{-1} \gamma_{kj}$, we obtain,
  \[ \mu_{ik}^{\bgamma_{ij}} = \mu_{jk} \mu_{ik} \mu_{jk}^{-\bgamma_{ij}} = \mu_{jk} \mu_{ik} (\bgamma_{jk} \mu_{kj}^{-1} \gamma_{kj} )^{\bgamma_{ij}} = \mu_{jk} \mu_{ik} \gamma_{jk}^{-\bgamma_{ij}}  \mu_{kj}^{-\bgamma_{ij}}  \gamma_{kj}^{\bgamma_{ij}}.   \]
  

The last two identities in (ii) follow from the relation $\gamma_{ik}\mu_{ij}\mu_{kj}=\mu_{kj}\mu_{ij}\gamma_{ik}$, and can be verified in a similar manner as the first two identities, with the difference that $\mu_{ik}$ is replaced with $\gamma_{ik}$.

(iii) To obtain the first identity, we make use of the relation $\mu_{ki}\mu_{kj}\mu_{ij} = \mu_{ij}\mu_{kj}\mu_{ki}$, as shown below:
\[ \mu_{ki}^{\mu_{ij}} = \mu_{ij}^{-1}\mu_{ki}\mu_{ij} =  \mu_{ij}^{-1} (\mu_{ki} \mu_{kj}\mu_{ij}) (\mu_{ij}^{-1} \mu_{kj}^{-1} \mu_{ij})  = \mu_{ij}^{-1} (\mu_{ij}\mu_{kj}\mu_{ki})\mu_{kj}^{-\mu_{ij}} = \mu_{kj}\mu_{ki}\mu_{kj}^{-\mu_{ij}}. \]
The second identity in (iii) follows from the same relation, as the following equalities demonstrate:
\begin{eqnarray*}
\mu_{ki}^{\mu_{ij}^{-1}} &=& \mu_{ij}\mu_{ki}\mu_{ij}^{-1} = \mu_{ij}(\mu_{kj}^{-1}\mu_{ij}^{-1}\mu_{ij}\mu_{kj})
	\mu_{ki}\mu_{ij}^{-1} = \mu_{kj}^{-\mu_{ij}^{-1}} (\mu_{ij}\mu_{kj} \mu_{ki})\mu_{ij}^{-1} \\
	&=&  \mu_{kj}^{-\mu_{ij}^{-1}} (\mu_{ki}\mu_{kj} \mu_{ij})\mu_{ij}^{-1} = \mu_{kj}^{-\mu_{ij}^{-1}}\mu_{ki}\mu_{kj}.
\end{eqnarray*}

The third and fourth identities in (iii) follow from the relation $\mu_{ki}\mu_{kj}\gamma_{ij}=\gamma_{ij}\mu_{kj}\mu_{ki}$. They can be derived in the same way as the first two identities, with the difference that $\mu_{ij}$ is replaced by $\gamma_{ij}$ and $\mu_{ij}^{-1}$ by $\bgamma_{ij}$. The fifth and sixth relations in (iii) can be obtained in the same way as the first two relations in (iii), by replacing $\mu_{ki}$ with $\gamma_{ki}$ and making use of the relation $\gamma_{ki}\mu_{kj}\mu_{ij} = \mu_{ij}\mu_{kj}\gamma_{ki}$. 

(iv) To verify the first identity, we start off with the following sequence of equalities:
\begin{eqnarray*}
\mu_{jk}^{\mu_{ij}} &=& \mu_{ij}^{-1}\mu_{jk}\mu_{ij} = \mu_{ij}^{-1}\mu_{jk}(\mu_{ik}\mu_{ij}\mu_{ij}^{-1}\mu_{ik}^{-1})\mu_{ij} =
 \mu_{ij}^{-1} (\mu_{jk}\mu_{ik}\mu_{ij}) \mu_{ij}^{-1}\mu_{ik}^{-1}\mu_{ij}\\
 &=&\mu_{ij}^{-1}(\mu_{ij}\mu_{ik}\mu_{jk})(\mu_{ij}^{-1}\mu_{ik}^{-1}\mu_{ij}) =  \mu_{ik}\mu_{jk}\mu_{ik}^{-\mu_{ij}},
 \end{eqnarray*}
 where the fourth equality holds due to the relation $\mu_{ij}\mu_{ik}\mu_{jk} = \mu_{jk}\mu_{ik}\mu_{ij}$.
Using the first identity in (ii), we have that:
\[\mu_{ik}^{-\mu_{ij}}=\left(\mu_{ik}^{\mu_{ij}}\right)^{-1}=\left(\mu_{kj}^{\mu_{ij}}\mu_{ik}\mu_{kj}^{-1} \right)^{-1}=\mu_{kj}\mu_{ik}^{-1}\mu_{kj}^{-\mu_{ij}}. \]
Therefore,
\[ \mu_{jk}^{\mu_{ij}}=\mu_{ik}\mu_{jk}\mu_{ik}^{-\mu_{ij}}=\mu_{ik}\mu_{jk}\mu_{kj}\mu_{ik}^{-1}\mu_{kj}^{-\mu_{ij}}. \]
Similarly, 
\begin{eqnarray*}
\mu_{jk}^{\mu_{ij}^{-1}} &=&\mu_{ij}\mu_{jk}\mu_{ij}^{-1}= \mu_{ij}(\mu_{ik}^{-1}\mu_{ij}^{-1}\mu_{ij} \mu_{ik})\mu_{jk}\mu_{ij}^{-1} = \mu_{ik}^{-\mu_{ij}^{-1}}(\mu_{ij} \mu_{ik}\mu_{jk})\mu_{ij}^{-1} \\
&=&  \mu_{ik}^{-\mu_{ij}^{-1}}(\mu_{jk} \mu_{ik}\mu_{ij})\mu_{ij}^{-1}=  \mu_{ik}^{-\mu_{ij}^{-1}}\mu_{jk}\mu_{ik}.
\end{eqnarray*}
From the second identity in (ii), we have that
\[\mu_{ik}^{-\mu_{ij}^{-1}}=\left(\mu_{ik}^{\mu_{ij}^{-1}}\right)^{-1}=\left(\mu_{kj}^{-1}\mu_{ik}\mu_{kj}^{\mu_{ij}^{-1}} \right)^{-1}=\mu_{kj}^{-\mu_{ij}^{-1}}\mu_{ik}^{-1}\mu_{kj}, \]
and therefore,
\[\mu_{jk}^{\mu_{ij}^{-1}}
=\mu_{ik}^{-\mu_{ij}^{-1}}\mu_{jk}\mu_{ik}=\mu_{kj}^{-\mu_{ij}^{-1}}\mu_{ik}^{-1}\mu_{kj}\mu_{jk}\mu_{ik}. \]
Hence, the second identity in (iv) holds. 

The third and fourth identities in (iv) are obtained using the relation $\mu_{jk} \gamma_{kj} = \gamma_{jk} \mu_{kj}$ (this is relation~\eqref{eq:pureTwist}). Hence, $\mu_{jk}  =  \gamma_{jk} \mu_{kj} \bgamma_{kj}$, and we obtain the following:
\[  \mu_{jk}^{\gamma_{ij}} = (\gamma_{jk} \mu_{kj} \bgamma_{kj})^{\gamma_{ij}} = \gamma_{jk}^{\gamma_{ij}} \mu_{kj}^{\gamma_{ij}} \gamma_{kj}^{-\gamma_{ij}} \,\, \text{and} \,\,
\mu_{jk}^{\bgamma_{ij}} = (\gamma_{jk} \mu_{kj} \bgamma_{kj})^{\bgamma_{ij}}  = \gamma_{jk}^{\bgamma_{ij}} \mu_{kj}^{\bgamma_{ij}}\gamma_{kj}^{-\bgamma_{ij}}.  \]

Finally, the last two identities in (iv) can be verified by substituting $\mu_{jk}$ with $\gamma_{jk}$ in the proofs of the first two identities in (iv), and making use of relation $\mu_{ij} \mu_{ik} \gamma_{jk} = \gamma_{jk} \mu_{ik} \mu_{ij}$, which is relation~\eqref{eq:pureYB-mixt} from Theorem~\ref{VSPGn Gamma Presentation}.


(v) We start demonstrating the first identity: 
\begin{eqnarray*}
\gamma_{kj}^{\mu_{ij}} &=& \mu_{ij}^{-1}\gamma_{kj}\mu_{ij} = \mu_{ij}^{-1} (\gamma_{kj}\mu_{ij} \mu_{ik}) \mu_{ik}^{-1} = \mu_{ij}^{-1}(\mu_{ik}\mu_{ij} \gamma_{kj}) \mu_{ik}^{-1}\\
&=& \mu_{ik}^{\mu_{ij}}\gamma_{kj} \mu_{ik}^{-1}= \mu_{kj}^{\mu_{ij}}\mu_{ik}\mu_{kj}^{-1} \gamma_{kj}\mu_{ik}^{-1}.
\end{eqnarray*}
The last equality above holds due to the first identity in (ii). The second identity in (v) can be verified as demonstrated below:
\begin{eqnarray*}
\gamma_{kj}^{\mu_{ij}^{-1}} &=& \mu_{ij}\gamma_{kj}\mu_{ij}^{-1} = (\mu_{ik}^{-1}\mu_{ik})(\mu_{ij}\gamma_{kj}\mu_{ij}^{-1}) =  \mu_{ik}^{-1} (\mu_{ik} \mu_{ij}\gamma_{kj}) \mu_{ij}^{-1}\\
&=&  \mu_{ik}^{-1} (\gamma_{kj} \mu_{ij}\mu_{ik}) \mu_{ij}^{-1} = \mu_{ik}^{-1}\gamma_{kj}\mu_{ik}^{\mu_{ij}^{-1}}  = \mu_{ik}^{-1}\gamma_{kj}\mu_{kj}^{-1}\mu_{ik} \mu_{kj}^{\mu_{ij}^{-1}},
\end{eqnarray*}
where the last equality holds due to the second identity in (ii). This completes the proof of Lemma~\ref{VSPGn Conjugation}.
\end{proof}

We will now look into the structure of the virtual singular pure braid group $VSPG_n$, using Bardakov's approach in~\cite{Bardakov} for the case of the virtual pure braid group. 

For each $2 \le i \le n$, we define 
\begin{align*}
VS_{i-1} : =& \langle \mu_{1, i},\,  \mu_{2, i}, \, \dots, \, \mu_{i-1, i},\,  \mu_{i,1},\,  \mu_{i, 2},\,  \dots,\,  \mu_{i, i-1}; \\
& \gamma_{1, i}, \, \gamma_{2, i},\,  \dots, \, \gamma_{i-1, i}, \, \gamma_{i, 1},  \,\gamma_{i, 2}, \, \dots, \, \gamma_{i, i-1} \rangle
\end{align*}
to be the subgroup of $VSPG_n$ generated by the elements listed above. Notice that each $VS_{i-1}$ is also a subgroup of $VSPG_{i}$. We define $VS^*_{i-1}$ to be the normal closure of $VS_{i-1}$ in $VSPG_{i}$. We attempt to represent $VSPG_n$ as a semi-direct product using these normal subgroups. We first consider the cases where $n=2$ and $n=3$.

Note that for $n=2$, we have $VSPG_2=VS_1=VS^*_1$. By Theorem~\ref{VSPGn Gamma Presentation}, the group $VSPG_2$, and hence $VS_1^*$, is generated by $\mu_{12}, \mu_{21}$ and $\gamma_{12}, \gamma_{21}$ with relation $\mu_{12}\gamma_{21} = \gamma_{12}\mu_{21}$.

We have that $VS_1=\langle \mu_{12},\mu_{21},\gamma_{12},\gamma_{21} \rangle$ and $VS_2=\langle \mu_{13},\mu_{23},\mu_{31},\mu_{32},\gamma_{13},\gamma_{23},\gamma_{31},\gamma_{32} \rangle$. Hence, $VSPG_3$ is generated by the generators of its subgroups $VS_1$ and $VS_2$.

\begin{definition}
We will refer to a braid in $VSG_n$ written in terms of the generators $\mu_{ij}^{\pm 1}, \gamma_{ij}$ and $\bgamma_{ij}$ as a \textit{word}. A word that cannot be simplified is called a \textit{reduced word}. Specifically, a reduced word is a word that does not contain products of elements that are inverses of each other.

 Let $w(\mu_{12}, \gamma_{12})$ be the empty word or a reduced word in $VS_1$ beginning with a non-zero power of $\mu_{12}$ or $\gamma_{12}$. The \textit{reduced power of the generator $\mu_{32}$}  is defined as $\mu_{32}^{w(\mu_{12}, \gamma_{12})}$. 

Let $w(\mu_{21}, \gamma_{21})$ be the empty word or a reduced word in $VS_1$ beginning with a non-zero power of $\mu_{21}$ or $\gamma_{21}$. The \textit{reduced power of the generator $\mu_{31}$} is defined as $\mu_{31}^{w(\mu_{21}, \gamma_{21})}$. 

 Let $w(\gamma_{12}, \gamma_{21})$ be the empty word or a reduced word in $VS_1$ beginning with a non-zero power of $\gamma_{12}$ or $\gamma_{21}$. For all $\gamma_{ij} \in \{ \gamma_{13}, \gamma_{23}, \gamma_{31}, \gamma_{32} \}$, the \textit{reduced power of the generators $\gamma_{ij}$} are defined as $\gamma_{ij}^{w(\gamma_{12}, \gamma_{21})}$. \end{definition}

\begin{proposition} \label{VSPG3 Structure}
The group $VSPG_3$ can be represented as the semi-direct product
\[VSPG_3=VS_2^* \rtimes VSPG_2=VS_2^* \rtimes VS_1^*.\]
Moreover, the group $VS_2^*$ is generated by $\mu_{13},\mu_{23}$ and all reduced powers of $\mu_{31},\mu_{32}$ and of $\gamma_{13}, \gamma_{23}, \gamma_{31}, \gamma_{32}$. 
\end{proposition}

\begin{proof}
We first note that $VS_2^* \cap VSPG_2=\{1_3\}$. We know that $VSPG_3$ is generated by the generators of $VS_2$ and  those of $VSPG_2 = VS_1$. Let $\alpha \in VSPG_2$ and $\beta \in VS_2$. Then,
\[ \alpha \beta = (\alpha \beta \alpha^{-1}) \alpha \in VS_2^* \cdot VSPG_2. \]
Hence, every element in $VSPG_3$ can be written in the form $w_2w_1$, where $w_1 \in VSPG_2$ and $w_2 \in VS_2^*$, and therefore, $VSPG_3=VS_2^* \rtimes VSPG_2$. Recall also that $VSPG_2 = VS_1 = VS_1^*$.
 
We know that $VS_2^*=\langle \mu_{13}^w,\mu_{23}^w,\mu_{31}^w,\mu_{32}^w; \gamma_{13}^w,\gamma_{23}^w,\gamma_{31}^w,\gamma_{32}^w \rangle$, where $w \in VS_1$. Due to Lemma~\ref{VSPGn Conjugation}, the following relations hold in $VSPG_3$.
\begin{itemize}
\item Conjugations by $\mu_{12}$ and $\mu_{12}^{-1}$ on $\mu_{13}, \mu_{23}, \mu_{31}$:
\begin{multicols}{2}
\begin{enumerate}
\item[(i)] $\mu_{13}^{\mu_{12}} =\mu_{32}^{\mu_{12}}\mu_{13}\mu_{32}^{-1}$ 
\item[(ii)] $\mu_{31}^{\mu_{12}} =\mu_{32}\mu_{31}\mu_{32}^{-\mu_{12}}$ 
\item[(iii)] $\mu_{23}^{\mu_{12}} =\mu_{13}\mu_{23}\mu_{32}\mu_{13}^{-1}\mu_{32}^{-\mu_{12}}$ 
\item[(iv)] $\mu_{13}^{\mu_{12}^{-1}} =\mu_{32}^{-1}\mu_{13}\mu_{32}^{\mu_{12}^{-1}}$
\item[(v)] $\mu_{31}^{\mu_{12}^{-1}} =\mu_{32}^{-\mu_{12}^{-1}}\mu_{31}\mu_{32}$
\item[(vi)] $\mu_{23}^{\mu_{12}^{-1}} =\mu_{32}^{-\mu_{12}^{-1}}\mu_{13}^{-1}\mu_{32}\mu_{23}\mu_{13}$
\end{enumerate}
\end{multicols}

\item Conjugations by $\mu_{21}$ and $\mu_{21}^{-1}$ on $\mu_{13}, \mu_{23}, \mu_{32}$:
\begin{multicols}{2}
\begin{enumerate}
\item[(i)] $\mu_{23}^{\mu_{21}} =\mu_{31}^{\mu_{21}}\mu_{23}\mu_{31}^{-1}$ 
\item[(ii)] $\mu_{32}^{\mu_{21}} =\mu_{31}\mu_{32}\mu_{31}^{-\mu_{21}}$ 
\item[(iii)] $\mu_{13}^{\mu_{21}} =\mu_{23}\mu_{13}\mu_{31}\mu_{23}^{-1}\mu_{31}^{-\mu_{21}}$ 
\item[(iv)] $\mu_{23}^{\mu_{21}^{-1}} =\mu_{31}^{-1}\mu_{23}\mu_{31}^{\mu_{21}^{-1}}$
\item[(v)] $\mu_{32}^{\mu_{21}^{-1}} =\mu_{31}^{-\mu_{21}^{-1}}\mu_{32}\mu_{31}$
\item[(vi)] $\mu_{13}^{\mu_{21}^{-1}} =\mu_{31}^{-\mu_{21}^{-1}}\mu_{23}^{-1}\mu_{31}\mu_{13}\mu_{23}$
\end{enumerate}
\end{multicols}


\item Conjugations by $\gamma_{12}$ and $\bgamma_{12}$ on $\mu_{13}, \mu_{23}, \mu_{31}$:
\begin{multicols}{2}
\begin{enumerate}

\item[(i)] $\mu_{13}^{\gamma_{12}} = \gamma_{23}^{-\gamma_{12}} \mu_{32}^{-\gamma_{12}} \gamma_{32}^{\gamma_{12}} \mu_{13} \mu_{23}$

\item[(ii)] $\mu_{31}^{\gamma_{12}} =\mu_{32}\mu_{31}\mu_{32}^{-\gamma_{12}}$

\item[(iii)] $\mu_{23}^{\gamma_{12}} =\gamma_{23}^{\gamma_{12}}\mu_{32}^{\gamma_{12}}\gamma_{32}^{-\gamma_{12}}$
\item[(iv)] $\mu_{13}^{\bgamma_{12}} = \mu_{23}\mu_{13} \gamma_{23}^{-\bgamma_{12}} \mu_{32}^{- \bgamma_{12}} \gamma_{32}^{\gamma_{12}}$

\item[(v)] $\mu_{31}^{\bgamma_{12}} =\mu_{32}^{-\bgamma_{12}}\mu_{31}\mu_{32}$\

\item[(vi)] $\mu_{23}^{\bgamma_{12}} = \gamma_{23}^{\bgamma_{12}}\mu_{32}^{\bgamma_{12}}\gamma_{32}^{-\bgamma_{12}}$


\end{enumerate}
\end{multicols}


\item Conjugations by $\gamma_{21}$ and $\bgamma_{21}$ on $\mu_{13}, \mu_{23}, \mu_{32}$:
\begin{multicols}{2}
\begin{enumerate}

\item[(i)] $\mu_{23}^{\gamma_{21}} = \gamma_{13}^{-\gamma_{21}} \mu_{31}^{\ \gamma_{21}} \gamma_{31}^{\gamma_{21}} \mu_{23} \mu_{13}$

\item[(ii)] $\mu_{32}^{\gamma_{21}} =\mu_{31}\mu_{32}\mu_{31}^{-\gamma_{21}}$

\item[(iii)] $\mu_{13}^{\gamma_{21}} = \gamma_{13}^{\gamma_{21}} \mu_{31}^{\gamma_{21}} \gamma_{31}^{-\gamma_{21}}$
\item[(iv)] $\mu_{23}^{\bgamma_{21}} = \mu_{13} \mu_{23} \gamma_{13}^{-\bgamma_{21}} \mu_{31}^{\bgamma_{21}} \gamma_{31}^{\bgamma_{21}}$

\item[(v)] $\mu_{32}^{\bgamma_{21}} =\mu_{31}^{-\bgamma_{21}}\mu_{32}\mu_{31}$

\item[(vi)] $\mu_{13}^{\bgamma_{21}} =  \gamma_{13}^{\bgamma_{21}} \mu_{31}^{\bgamma_{21}} \gamma_{31}^{-\bgamma_{21}}$


\end{enumerate}
\end{multicols}

\item Conjugations by $\mu_{12}$ and $\mu_{12}^{-1}$ on $\gamma_{13}, \gamma_{23}, \gamma_{31}$, and $\gamma_{32}$:
\begin{multicols}{2}
\begin{enumerate}
\item[(i)] $\gamma_{13}^{\mu_{12}} =\mu_{32}^{\mu_{12}}\gamma_{13}\mu_{32}^{-1}$
\item[(ii)] $\gamma_{31}^{\mu_{12}} =\mu_{32}\gamma_{31}\mu_{32}^{-\mu_{12}}$
\item[(iii)] $\gamma_{23}^{\mu_{12}} =\mu_{13}\gamma_{23}\mu_{32}\mu_{13}^{-1}\mu_{32}^{-\mu_{12}}$
\item[(iv)] $\gamma_{32}^{\mu_{12}}
=\mu_{32}^{\mu_{12}}\mu_{13}\mu_{32}^{-1}\gamma_{32}\mu_{13}^{-1}$
\item[(v)] $\gamma_{13}^{\mu_{12}^{-1}} =\mu_{32}^{-1}\gamma_{13}\mu_{32}^{\mu_{12}^{-1}}$
\item[(vi)] $\gamma_{31}^{\mu_{12}^{-1}} =\mu_{32}^{-\mu_{12}^{-1}}\gamma_{31}\mu_{32}$
\item[(vii)] $\gamma_{23}^{\mu_{12}^{-1}} =\mu_{32}^{-\mu_{12}^{-1}}\mu_{13}^{-1}\mu_{32}\gamma_{23}\mu_{13}$
\item[(viii)] $\gamma_{32}^{\mu_{12}^{-1}} =\mu_{13}^{-1}\gamma_{32} \mu_{32}^{-1}\mu_{13}\mu_{32}^{\mu_{12}^{-1}}$
\end{enumerate}
\end{multicols}

\item Conjugations by $\mu_{21}$ and $\mu_{21}^{-1}$ on $\gamma_{13}, \gamma_{23}, \gamma_{32}$, and $\gamma_{31}$:
\begin{multicols}{2}
\begin{enumerate}
\item[(i)] $\gamma_{23}^{\mu_{21}} =\mu_{31}^{\mu_{21}}\gamma_{23}\mu_{31}^{-1}$
\item[(ii)] $\gamma_{32}^{\mu_{21}} =\mu_{31}\gamma_{32}\mu_{31}^{-\mu_{21}}$
\item[(iii)] $\gamma_{13}^{\mu_{21}} =\mu_{23}\gamma_{13}\mu_{31}\mu_{23}^{-1}\mu_{31}^{-\mu_{21}}$
\item[(iv)] $\gamma_{31}^{\mu_{21}} =\mu_{31}^{\mu_{21}}\mu_{23}\mu_{31}^{-1}\gamma_{31}\mu_{23}^{-1}$
\item[(v)] $\gamma_{23}^{\mu_{21}^{-1}} =\mu_{31}^{-1}\gamma_{23}\mu_{31}^{\mu_{21}^{-1}}$
\item[(vi)] $\gamma_{32}^{\mu_{21}^{-1}} =\mu_{31}^{-\mu_{21}^{-1}}\gamma_{32}\mu_{31}$
\item[(vii)] $\gamma_{13}^{\mu_{21}^{-1}} =\mu_{31}^{-\mu_{21}^{-1}}\mu_{23}^{-1}\mu_{31}\gamma_{13}\mu_{23}$
\item[(viii)] $\gamma_{31}^{\mu_{21}^{-1}} =\mu_{23}^{-1}\gamma_{31} \mu_{31}^{-1}\mu_{23}\mu_{31}^{\mu_{21}^{-1}}$
\end{enumerate}
\end{multicols}
\end{itemize}

Therefore, we can express $\mu_{13}^w$ and $\mu_{23}^w$ in terms of the conjugation of $\mu_{31}$ and $\mu_{32}$ up to multiplication with other words. Similarly, we can express $\mu_{31}^{\mu_{12}^{\pm 1}}$, $\mu_{31}^{\gamma_{12}^{\pm 1}}$, $\mu_{32}^{\mu_{21}^{\pm 1}}$, and $\mu_{32}^{\gamma_{21}^{\pm 1}}$ in terms of reduced powers of $\mu_{31}$ and $\mu_{32}$.
Moreover, we can also express $\gamma_{ij}^{\mu_{kl}^{\pm 1}}$, where $\gamma_{ij} \in \{ \gamma_{13},\gamma_{23},\gamma_{31}, \gamma_{32}\}$ and $\mu_{kl} \in \{ \mu_{12}, \mu_{21} \}$ in terms of some reduced powers of $\mu_{31}$ and $\mu_{32}$ up to multiplication with other words.
Hence, $VS_2^*$ is generated by $\mu_{13}, \mu_{23}$ and all reduced powers of $\mu_{31},\mu_{32}$ and $\gamma_{13}, \gamma_{23}, \gamma_{31}, \gamma_{32}$. 
\end{proof}

The following is a direct consequence of Proposition~\ref{VSPG3 Structure}.
\begin{corollary}
Every word in $VSPG_3$ can be uniquely written in the form $w = w_2w_1$, where $w_1$ is a reduced word over the alphabet $\{\mu_{12}^{\pm 1}, \mu_{21}^{\pm 1}, \gamma_{12}^{\pm 1}, \gamma_{21}^{\pm 1} \}$ and $w_2$ is a reduced word over the alphabet $\{ \mu_{13}^{\pm1}, \mu_{23}^{\pm 1}, \mu_{31}^{\pm w(\mu_{21}, \gamma_{21})}, \mu_{32}^{\pm w(\mu_{12}, \gamma_{12})}, \gamma_{i3}^{\pm w(\gamma_{12}, \gamma_{21})}, \gamma_{3i}^{\pm w(\gamma_{12}, \gamma_{21})} \, | \, i = 1, 2   \}$.
\end{corollary}

\begin{definition} 
Let $2 \le k \le n$ and $1 \le j \le k-1$. The \textit{reduced power of the generator $\mu_{kj} \in VSPG_k$} is defined as $\mu_{kj}^{w(\mu_{ij}, \gamma_{ij})}$, where $w(\mu_{ij}, \gamma_{ij})$ is the empty word or a word in $VSPG_{k-1}$ beginning with a non-zero power of some generator $\mu_{ij}$ or $\gamma_{ij}$, where $1 \le i \le k-1$ and $i \neq j$. The \textit{reduced power of the generators $\gamma_{kj}, \gamma_{jk} \in VSPG_k$} are defined as $\gamma_{kj}^{w(\gamma_{ij}, \gamma_{ji})}$ and $\gamma_{jk}^{w(\gamma_{ij}, \gamma_{ji})}$, respectively, where $w(\gamma_{ij}, \gamma_{ji})$ is the empty word or a word in $VSPG_{k-1}$ beginning with a non-zero power of some generator $\gamma_{ij}$ or $\gamma_{ji}$, where $1 \le i \le k-1$ and $i \neq j$.
\end{definition}

\begin{proposition} \label{Generators of VSi*}
Let $3 \le k \le n$. The subgroup $VS_{k-1}^*$ of $VSPG_{k}$ is generated by $\mu_{jk}$ for all $1 \le j \le k-1$ and all reduced powers of the generators $\mu_{kj}, \gamma_{kj}$ and $\gamma_{jk}$ for all $1 \le j \le k-1$. 
\end{proposition}

\begin{proof}
The statement is a consequence of Lemma~\ref{VSPGn Conjugation} and it is proved in a similar manner as Proposition~\ref{VSPG3 Structure}.
\end{proof}

We are now ready to state and prove our main theorem about the structure of the group $VSPG_n$.

\begin{theorem} \label{Semi Direct Product of VSPGn} 
The group $VSPG_n$, $n \ge 2$, is representable as the semi-direct product
\[VSPG_n=VS_{n-1}^* \rtimes VSPG_{n-1}=VS_{n-1}^* \rtimes \left(VS_{n-2}^* \rtimes \left( \cdots \rtimes \left(VS_2^* \rtimes VS_1^* \right)  \cdots \right) \right), \]
where $VS_{k-1}^*$, for all $3 \le k \le n$, are infinitely generated subgroups of $VSPG_n$ and $VS_1^*$ is a subgroup of rank $4$.
\end{theorem}

\begin{proof} 
The first decomposition as the semi-direct product $VSPG_n = VS_{n-1}^* \rtimes VSPG_{n-1}$ is obvious, since $VS_{n-1}^* \cap VSPG_{n-1} = \{1_n\}$ and $VSPG_n = VS_{n-1}^* \cdot VSPG_{n-1}$. [Note that the latter equality holds since $VSPG_n$ is generated by the union of the generators of $VS_{n-1}$ and $VSPG_{n-1}$, and since if $\alpha \in VSPG_{n-1}$ and $\beta \in VS_{n-1}$, then $\alpha \beta = (\alpha \beta \alpha^{-1}) \alpha \in VS_{n-1}^* \cdot VSPG_{n-1}$.]

We prove the second decomposition by induction on $n$.
Recall that $VSPG_2 = VS_1^*$, so the statement holds trivially for $n=2$. By Proposition~\ref{VSPG3 Structure}, $VSPG_3 = VS_2^* \rtimes VS_1^*$, and thus the statement is true for $n=3$, as well.
For the induction hypothesis, we assume that the group $VSPG_{n-1}$ can be represented as the following semi-direct product:
\begin{eqnarray*}
VSPG_{n-1} = VS_{n-2}^* \rtimes \left(VS_{n-3}^* \rtimes \left( \cdots \rtimes \left(VS_2^* \rtimes VS_1^* \right) \cdots  \right) \right). 
\end{eqnarray*}
Since $VSPG_n = VS_{n-1}^* \rtimes VSPG_{n-1}$, we obtain that 
\begin{eqnarray*}
VSPG_n &=&VS_{n-1}^* \rtimes VSPG_{n-1}\\
&=& VS_{n-1}^* \rtimes  \left(VS_{n-2}^* \rtimes \left(VS_{n-3}^* \rtimes \left( \cdots \rtimes \left(VS_2^* \rtimes VS_1^* \right) \cdots  \right) \right) \right).
\end{eqnarray*}

Finally, it is quite evident that for each $3 \le k \le n$, $VS_{k-1}^*$ is an infinitely generated subgroup of $VSPG_n$. We also know that $VS_1^*$ is generated by $\mu_{12}, \mu_{21}, \gamma_{12}, \gamma_{21}$, and thus it is a subgroup of rank $4$.
\end{proof}

As a consequence of Theorem~\ref{Semi Direct Product of VSPGn}, we obtain a normal form of words in $VSPG_n$ and $VSG_n$.
\begin{corollary}
Every element $\beta \in VSG_n$ can be uniquely written in the form 
\[ \beta = w_{n-1}w_{n-2} \cdots w_2w_1 \lambda, \hspace{1cm} \text{where} \, \, w_k \in VS_k^*, \,\, \lambda \in \Lambda_n,\]
where each $w_k$ is a reduced word over the alphabet consisting of the generators $\mu_{j,k+1}^{\pm 1}$ and the reduced powers of the generators $\mu_{k+1,j}^{\pm 1}, \gamma_{k+1,j}^{\pm 1}, \gamma_{j,k+1}^{\pm 1}$, for $1 \le j \le k$, and where  $\Lambda_n$ is a Schreier set of coset representatives of $VSPG_n$ in $VSG_n$, as described in the proof of Theorem~\ref{VSPGn Gamma Presentation}.
\end{corollary}

\begin{proof}
Since every element $\beta \in VSG_n$ belongs to a unique right coset of $VSPG_n$ in $VSG_n$, the braid $\beta$ can be uniquely written in the form $\beta = w \lambda$, where $w \in VSPG_n$ and $\lambda \in \Lambda_n$. By Theorem~\ref{Semi Direct Product of VSPGn}, the braid $w$ can be uniquely written in the form 
\[ w= w_{n-1}w_{n-2} \cdots w_2w_1, \,\,\, \text{where} \,\, w_k \in VS_k^*.\]
Since each $w_k$ is a reduced word in the generators of $VS_k^*$ and their inverses, the remaining part of the statement follows from Proposition~\ref{Generators of VSi*} and the fact that $VS_1^*$ is generated by $\mu_{12}, \mu_{21}$ and $\gamma_{12}, \gamma_{21}$. 
\end{proof}


\end{document}